\documentclass[12pt]{article}
\usepackage{amssymb,amsmath, amsthm}
\usepackage{amssymb}
\usepackage{framed}
\usepackage{mathrsfs}
\usepackage{enumerate}
\usepackage{tikz}
\usetikzlibrary{matrix}
\usepackage[all]{xy}
\usetikzlibrary{shapes}
\usepackage{hyperref}
\usepackage{multirow}
\usepackage{bm}
\usepackage[boxruled]{algorithm2e}
\usepackage{algpseudocode}
\usepackage{cite}
\usepackage{enumitem}
\setenumerate[0]{leftmargin=*}

\numberwithin{equation}{section}

\DeclareMathOperator*{\dom}{dom\,}
\DeclareMathOperator*{\lip}{lip\,}

\newcommand{\inclu}[0] {\ar@{^{(}->}}

\newcommand{\PP}{\mathbb{P}}
\newcommand{\RR}{\mathbb{R}}

\newcommand{\R}{{\bf R}}
\newcommand{\EE}{{\mathbb E}}
\newcommand{\cb}{\textup{\texttt{proxBoost}}}
\newcommand{\cberm}{\textup{\texttt{BoostERM}}}
\newcommand{\cbermc}{\textup{\texttt{BoostERMC}}}
\newcommand{\pboost}{\textup{\texttt{proxBoost}}\,}
\newcommand{\pboostn}{\textup{\texttt{proxBoost}}}
\newcommand{\rfg}{\textup{\texttt{RobustGap}}}
\newcommand{\cM}{\mathcal{M}}
\newcommand{\cD}{\mathcal{D}}
\newcommand{\alg}{\textup{\texttt{Alg}}}
\newcommand{\ralg}{\textup{\texttt{Alg-R}}}
\newcommand{\balg}{\textup{\texttt{BoostAlg}}}
\newcommand{\balgc}{\textup{\texttt{BoostAlgC}}}

\newcommand{\erm}{\textup{\texttt{ERM}}}
\newcommand{\ermc}{\textup{\texttt{ERMC}}}
\newcommand{\rermc}{\textup{\texttt{ERM-RC}}}
\newcommand{\rerm}{\textup{\texttt{ERM-R}}}
\newcommand{\cX}{\mathcal{X}}

\newcommand{\ext}{\textup{\texttt{Extract}}}

\newcommand{\cG}{\mathcal{G}}

\newcommand{\cO}{\mathcal{O}}
\newcommand{\cP}{\mathcal{P}}
\newcommand{\cI}{\mathcal{I}}

\usepackage[margin=1.1in]{geometry}

\newcommand{\argmin}{\operatornamewithlimits{argmin}}

\SetAlgoSkip{bigskip}
\SetKwInput{Input}{Input\hspace{34pt}{ }}
\SetKwInput{Output}{Output\hspace{27pt}{ }}
\SetKwInput{Initialize}{Initialize\hspace{2pt}{ }}
\SetKwInput{Parameters}{Parameters\hspace{2pt}{ }}

\newtheorem{theorem}{Theorem}[section]

\newtheorem{lemma}[theorem]{Lemma}

\newtheorem{defn}[theorem]{Definition}
\newtheorem{corollary}[theorem]{Corollary}
\newtheorem{assumption}[theorem]{Assumption}

\newtheorem{example}{Example}[section]

\usepackage{mathtools}

\title{From low probability to high confidence in stochastic convex optimization}

\author{Damek Davis\thanks{School of ORIE, Cornell University,
		Ithaca, NY 14850, USA;
		\texttt{people.orie.cornell.edu/dsd95/}.}\and Dmitriy Drusvyatskiy\thanks{Department of Mathematics, U. Washington,
		Seattle, WA 98195; Microsoft Research, Redmond, WA 98052;  \texttt{www.math.washington.edu/{\raise.17ex\hbox{$\scriptstyle\sim$}}ddrusv}. Research of Drusvyatskiy was supported by the NSF DMS   1651851 and CCF 1740551 awards.}\and Lin Xiao\thanks{Microsoft Research,
		Redmond, WA, USA; \texttt{www.microsoft.com/en-us/research/people/lixiao/}} \and Junyu Zhang\thanks{Department of Industrial and Systems Engineering, University
		of Minnesota, Minneapolis, MN, USA; \texttt{zhan4393@umn.edu}}}

\begin{document}
\date{}
\maketitle

\begin{abstract}


Standard results in stochastic convex optimization bound the number of samples that an algorithm needs to generate a point with small function value in expectation. More nuanced \emph{high probability} guarantees are rare, and typically either rely on ``light-tail'' noise assumptions or exhibit worse sample complexity. In this work, we show that a wide class of stochastic optimization algorithms for strongly convex problems can be augmented with high confidence bounds at an overhead cost that is only logarithmic in the confidence level and polylogarithmic in the condition number. The procedure we propose, called \pboostn, is elementary and builds on two well-known ingredients: robust distance estimation and the proximal point method. We discuss consequences for both streaming (online) algorithms and offline algorithms based on empirical risk minimization. 
\end{abstract}

\section{Introduction}
Stochastic convex optimization lies at the core of modern statistical and machine learning. Standard results in the subject bound the number of samples that an algorithm needs to generate a point with small function value in {\em expectation}.
%
Specifically, consider 
\begin{equation}\label{eqn:stoch-opt-prob}
    \min_{x}~ f(x):=\EE_{z\sim\cP}[f(x,z)],
\end{equation}
where the random variable $z$ follows a fixed unknown distribution~$\cP$
and $f(\cdot,z)$ is convex for almost every $z\sim\cP$. 
Given a small tolerance $\epsilon>0$, stochastic gradient methods typically 
produce a point $x_{\epsilon}$ satisfying
\[
    \EE[f(x_{\epsilon})] - \min f \leq \epsilon.
\]
The cost of the algorithms, measured by the required number of 
stochastic (sub-)gradient evaluations,
is $\mathcal{O}(1/\epsilon^2)$ or $\mathcal{O}(1/\epsilon)$ if $f$ is strongly convex (e.g., \cite{complexity,MR1167814,ghadimi2013optimal}).

In this paper, we are  interested in procedures that can produce 
an approximate solution with \emph{high probability}, meaning
a point $x_{\epsilon,p}$ satisfying
\begin{equation}\label{eqn:high_prob_bound}
\PP(f(x_{\epsilon, p})-\min f\leq \epsilon)\geq 1-p,
\end{equation}
where $p>0$ can be arbitrarily small.
By Markov's inequality, one can guarantee~\eqref{eqn:high_prob_bound}
by generating a point $x_{\epsilon,p}$ satisfying 
$\EE[f(x_{\epsilon,p})]-f^*\leq p\epsilon$,
e.g., by using standard stochastic gradient methods.
However, the resulting sample complexity can be very high for small~$p$ 
with the typical scaling of $\mathcal{O}(1/(p\epsilon))$ 
or $\mathcal{O}(1/(p\epsilon)^2)$.
Existing literature does provide a path to reducing the dependence of the sample complexity on~$p$
to $\log(1/p)$, but this usually comes with cost of either
worse dependence on~$\epsilon$ 
(e.g., \cite{BousquetElisseeff2002,NesterovVial2008,shalev2009stochastic})
or more restrictive sub-Gaussian assumptions on the stochastic gradient noise
(e.g. \cite{MR2486041,MR3353214,MR3023780,ghadimi2013optimal}). 

We aim to develop \emph{generic} low-cost procedures that equip stochastic optimization algorithms with high confidence guarantees, without making restrictive noise assumptions.
Consequently,  it will be convenient to treat such algorithms as black boxes. More formally, suppose that the function $f$ may only be accessed through a {\em minimization oracle} $\mathcal{M}(f,\epsilon)$, which on input $\epsilon>0$, returns a point $x_{\epsilon}$ satisfying the low confidence bound
\begin{equation}\label{eqn:conf_bound}
\PP(f(x_{\epsilon})-\min f\leq \epsilon)\geq \frac{2}{3}.
\end{equation}
(By Markov's inequality, minimization oracles arise from any algorithm that can generate $x_{\epsilon}$ satisfying $\EE f(x_{\epsilon})-\min f\leq \epsilon/3$.)
Let $\mathcal{C}_{\cM}(f, \epsilon)$ denote the cost of the oracle call
$\cM(f, \epsilon)$. Given a minimization oracle and its cost, 
we investigate the following question: 
\begin{quote}
Is there a procedure within this oracle model of computation that returns a point $x_{\epsilon,p}$ satisfying the high confidence bound \eqref{eqn:high_prob_bound} at a total cost that is only a ``small'' multiple of $\mathcal{C}_{\cM}(f, \epsilon)\cdot\log(\frac{1}{p})$?
\end{quote}
We will see that when~$f$ is strongly convex, the answer is yes for a wide class of oracles $\mathcal{M}(f,\epsilon)$. To simplify discussion, 
suppose $f$ is $\mu$-strongly convex and $L$-smooth (differentiable with $L$-Lipschitz continuous gradient). Then the cost $\mathcal{C}_{\cM}(f,\epsilon)$ typically depends on the condition number $\kappa:=L/\mu\gg 1$, as well as scale sensitive quantities such as initialization quality and upper bound on the gradient variances, etc.  
The procedures introduced in this paper execute the minimization oracle multiple
times in order to boost its confidence, with the total cost on the order of 
$$ \log\left(\frac{\log(\kappa)}{p}\right)\log(\kappa)\cdot \mathcal{C}_{\cM}\left(f,\tfrac{\epsilon}{\log(\kappa)}\right).$$
Thus, high probability bounds are achieved with a small cost increase, which depends only logarithmically on $1/p$ and polylogarithmically on the condition number $\kappa$.

Before introducing our approach, we discuss two techniques for boosting the confidence of a minimization oracle, both of which have limitations. As a first approach, one may query the oracle $\mathcal{M}(f,\epsilon)$ multiple times and pick the ``best" iterate from the batch. This is a flawed strategy since often one cannot test which iterate is ``best" without increasing sample complexity. To illustrate, consider estimating the expectation $f(x)=\EE_z \left[f(x,z)\right]$ to $\epsilon$-accuracy for a fixed point $x$. This task amounts to mean estimation, which requires on the order of $1/\epsilon^2$ samples, even under sub-Gaussian assumptions \cite{MR3052407}. In this paper, the cost $\mathcal{C}_{\cM}(f,\epsilon)$ typically scales at worst as $1/\epsilon$, and therefore mean estimation would significantly degrade the overall sample complexity.

The second approach leverages the fact that, with strong convexity, 
\eqref{eqn:conf_bound} implies 
 $$\PP(\|x_\epsilon-\bar x\|\leq \sqrt{2\epsilon/\mu})\geq \frac{2}{3},$$
 where $\bar x$ is the minimizer of $f$. Given this bound, one may apply the \emph{robust distance estimation} technique of~\cite[p. 243]{complexity} and~\cite{hsu2016loss} to choose a point near $\bar x$: Run $m$ trials of $\mathcal{M}(f,\epsilon)$ and find one iterate $x_{i^*}$ around which the other points ``cluster''. Then the point $x_{i^*}$ will be within a distance of $\sqrt{18\epsilon/\mu}$ from $\bar x$ with probability $1-\exp(-m/18)$. The downside of this strategy is that when converting naively back to function values, the suboptimality  gap becomes $f(x_{i^*})-\min f \leq \frac{L}{2}\|x_{i^*} - \bar x\|^2 \leq 9\kappa\epsilon$. Thus the function gap at $x_{i^*}$ may be significantly larger than the expected function gap at $x_{\epsilon}$, by  a factor of the condition number. Therefore, robust distance estimation exhibits a trade-off between robustness and efficiency.

The robustness/efficiency trade-off disappears for perfectly conditioned losses. Therefore, it appears plausible that one might avoid the $\kappa$ factor through a continuation procedure that solves a sequence of nearby, better conditioned problems. This is the strategy we explore here. The \pboost procedure embeds robust distance estimation inside a proximal point method. It begins by declaring the initial point $x_0$ to be the output of the robust distance estimator for minimizing~$f$. Then the better conditioned function 
$$f^t(x):=f(x)+\frac{\mu 2^{t}}{2}\|x-x_t\|^2,$$
is formed and the next iterate $x_{t+1}$ is declared to be the output of the robust distance estimator for minimizing~$f^t$. Since the conditioning of $f^t$ rapidly improves with $t$, the robust distance estimator becomes more efficient as the counter~$t$ grows.  

The \pboost method can be applied to a wide class of stochastic minimization oracles, including both streaming algorithms (e.g., stochastic gradient methods) and offline methods such as empirical risk minimization (ERM).
We now illustrate the consequences of \pboost for solving the problem~\eqref{eqn:stoch-opt-prob} using these two types of oracles.

\subsection{Streaming Oracles}\label{sec:subsecintro1}
Stochastic gradient methods can be treated as minimization oracles $\mathcal{M}(f,\epsilon)$ whose cost $\mathcal{C}_{\cM}(f,\epsilon)$ are measured by the number stochastic gradients needed to reach functional accuracy~$\epsilon$ in expectation.
An algorithm with minimal such cost was proposed by Ghadimi and Lan \cite{ghadimi2013optimal}. It generates a point $x_\epsilon$ satisfying $\EE\left[f(x_\epsilon) - \min f\right] \leq \epsilon$ with
\begin{equation}\label{eqn:Lefficiency_estimateintrolan}
\mathcal{O}\left(\sqrt{\kappa}\ln\left(\frac{\Delta_{\rm in}}{\epsilon}\right)+\frac{\sigma^2}{\mu \epsilon}\right)
\end{equation}
stochastic gradient evaluations, where the quantity $\sigma^2$ is an upper bound on the variance of the stochastic gradient estimator $\nabla f(x,z)$ and $\Delta_{\rm in}$ is a known upper bound on the initial function gap $\Delta_{\rm in}\geq f(x_0)-f^*$. A simpler algorithm with a similar efficiency estimate was recently presented by Kulunchakov and Mairal \cite{kulunchakov2019estimate}, and was based on estimate sequences. Aybat et al. \cite{aybat2019universally}  developed an algorithm with similar efficiency, but in contrast to previous work, it does not require the variance $\sigma^2$ and the initial gap $\Delta_{\rm in}$ as inputs.

It is intriguing to ask if one can equip the stochastic gradient method and its accelerated variant with high confidence guarantees. In their original work \cite{ghadimi2013optimal,MR3023780}, Ghadimi and Lan  provide an affirmative answer under the additional assumption that the stochastic gradient estimator has light tails. 
The very recent work of Juditsky-Nazin-Nemirovsky-Tsybakov \cite{juditsky2019algorithms} shows that one can avoid the light tail assumption for the basic stochastic gradient method, and for mirror descent more generally, by truncating the gradient estimators. High confidence bounds for the accelerated method, without light tail assumptions, remain open.

In this work, the optimal method of~\cite{ghadimi2013optimal} will be used as a minimization oracle within \pboostn, allowing us to nearly match the efficiency estimate~\eqref{eqn:Lefficiency_estimateintrolan} without ``light-tail" assumptions. Equipped with this oracle, \pboost returns a point $x_{\epsilon,p}$ satisfying~\eqref{eqn:high_prob_bound} 
and the overall cost of the procedure is
$$\widetilde{\mathcal{O}}\left(\log\left(\frac{1}{p}\right)\left(\sqrt{\kappa}\ln\left(\frac{\Delta_{\rm in}}{\epsilon}\vee \kappa\right)+\frac{\sigma^2}{\mu \epsilon}\right)\right).$$
Here, $\widetilde{\mathcal{O}}(\cdot)$ only suppresses logarithmic dependencies in $\kappa$; see Section~\ref{sec:conseq_approx} for a precise guarantee. Thus for small $\epsilon$, the sample complexity of the robust procedure is roughly $\log(1/p)$ times the efficiency estimate \eqref{eqn:Lefficiency_estimateintrolan} of the low-confidence algorithm.  

\subsection{Empirical Risk Minimization Oracles}\label{sec:subsecintro2}
An alternative approach to streaming algorithms, such as the stochastic gradient method, is based on empirical risk minimization (ERM) or sample average approximation (SAA) \cite{ShapiroNemirovski2005}. Namely, we draw i.i.d.\ samples $z_1,\ldots, z_n\sim \cP$ and minimize the empirical average
\begin{equation}
    \min_{x}~ f_S(x):=\frac{1}{n}\sum_{i=1}^n f(x,z_i).
\end{equation}
A key question is to determine the number $n$ of samples that would ensure that the minimizer $x_S$ of the empirical risk $f_S$ has low generalization error $f(x_S)-\min f$, with reasonably high probability. There is a vast literature on this subject; see for example \cite{hsu2016loss,bartlett2002rademacher,shalev2009stochastic,shalev2014understanding}. We build here on the work of Hsu-Sabato \cite{hsu2016loss}, who focused on high confidence guarantees for nonnegative losses $f(x,z)$. They showed that the empirical risk minimizer $x_S$ yields a robust distance estimator of the true minimizer of $f$.
As a consequence they deduced that ERM can find a point $x_S$ satisfying the relative error guarantee
$$\PP\bigl[f(x_S)\leq (1+\gamma)f^*\bigr]\geq 1-p,$$
with the sample complexity $n$ on the order of 
$$\mathcal{O}\left(\log\left(\frac{1}{p}\right)\cdot\frac{\hat{\kappa} \, \kappa}{\gamma}\right).$$
Loosely speaking, here $\kappa$ and $\hat{\kappa}$ are the condition numbers of  $f$ and $f_S$, respectively. By embedding ERM within \pboostn, we obtain the much better sample complexity
$$\widetilde{\mathcal{O}}\left(\log\left(\frac{1}{p}\right)\left(\frac{\hat \kappa}{\gamma}+\hat \kappa\right)\right),$$
where the symbol $\widetilde{\mathcal{O}}$ only suppresses polylogarithmic dependence on $\kappa$ and $\hat \kappa$. See Section~\ref{sec:cons_ERM} for the precise sample complexity guarantee.

\subsection{Convex composite optimization}
The results, previewed so far rely on the assumption that $f$ is strongly convex and smooth. These techniques can not directly accommodate constraints or nonsmooth regularizers. 
To illustrate the difficulty, consider the convex composite optimization problem
\begin{equation}\label{eqn:comp_general}
\min_{x} f(x)=g(x)+h(x),
\end{equation}
where $g\colon\R^d\to\R$ is smooth and strongly convex 
and $h\colon\R^d\to\R\cup\{+\infty\}$ is an arbitrary closed convex function. 
For example, a constrained optimization problem can be modeled by setting $h$ to be zero on the feasible region and plus infinity elsewhere.
The approach for the unconstrained problems, outlined previously, heavily relies on the fact that the function gap $f(x)-\min f$ and the squared distance to the solution $\|x-\bar x\|^2$ are proportional up to multiplication by the condition number. The analogous statement for the composite setting~\eqref{eqn:comp_general} is decisively false. In particular,  it is unclear how to turn low probability guarantees on the function gap $f(x)-\min f$ to high probability outcomes, even if one was willing to degrade the accuracy by the condition number of $g$.

In the last section of the paper, we resolve this apparent difficulty and thereby generalize the $\cb$ framework to the entire composite problem class \eqref{eqn:comp_general}. The key tool is a new robust distance estimation technique
 for convex composite problems, which may be of independent interest. Consequences for regularized empirical risk minimization and proximal streaming algorithms, in the spirit of Sections \ref{sec:subsecintro1} and \ref{sec:subsecintro2}, follow immediately.

\subsection{Related literature}
Our paper rests on two pillars: the proximal point method and robust distance estimation. The two techniques have been well studied in the optimization and statistics literature respectively. The proximal point method was introduced by Martinet \cite{MR0290213,MR0298899} and further popularized by Rockafellar \cite{MR0410483}. This construction is also closely related to the smoothing function of Moreau \cite{MR0201952}. Recently, there has been a renewed interest in the proximal point method, most notably due to its uses in accelerating variance-reduction methods for minimizing finite sums of convex functions \cite{catalyst,frostig,conjugategradient,accsdca}. The proximal point method has also featured prominently as a guiding principle in nonconvex optimization, with the works of \cite{asi2019importance,asi2018stochastic,duchi2018stochastic,davis2019stochastic,prixm_guide_subgrad}. 
The stepsize schedule we use within the proximal point method is geometrically decaying, in contrast to the more conventional polynomially decaying schemes. Geometrically decaying schedules for subgradient methods were first used by Goffin \cite{goffin} and have regained some attention recently due to their close connection to the popular step-decay schedule in stochastic optimization \cite{ge2019step,aybat2019universally,xu2016accelerated,yang2018does}.

Robust distance estimation has a long history. The estimator we use was first introduced in \cite[p. 243]{complexity}, and can be viewed as a multivariate generalization of the median of means estimator \cite{MR1688610,MR855970}. Robust distance estimation was further investigated in \cite{hsu2016loss} with a focus on high probability guarantees for empirical risk minimization. A different generalization based on the geometric median was studied in \cite{MR3378468}.  Other recent articles related to the subject include median of means tournaments \cite{lugosi2016risk}, robust multivariate mean estimators \cite{joly2017estimation,lugosi2019sub}, and  bandits with heavy tails \cite{bubeck2013bandits}.

One of the main applications of our techniques is to streaming algorithms. Most currently available results that establish high confidence convergence guarantees make sub-Gaussian assumptions on the stochastic gradient estimator \cite{MR2486041,MR3353214,MR3023780,ghadimi2013optimal}. More recently, there has been renewed interest in obtaining robust guarantees without the light-tails assumption. For example, the two works \cite{chen2017distributed,pmlr-v80-yin18a} make use of the geometric median of means technique to robustly estimate the gradient in distributed optimization.
A different technique was recently developed by Juditsky et al. \cite{juditsky2019algorithms}, where the authors  establish high confidence guarantees for mirror descent type algorithms by truncating the gradient.

\bigskip

The outline of the paper is as follows. Section~\ref{sec:prob_set} presents the problem setting and robust distance estimation. Section \ref{sec:main_res} develops the \pboost procedure. Section~\ref{sec:cons_ERM} presents consequences for empirical risk minimization, while Section~\ref{sec:conseq_approx} discusses consequences for streaming algorithms, both in the strongly convex and smooth setting. The final Section~\ref{sec:conv_comp_ext} extends the aforementioned techniques to convex composite problems.


\section{Problem setting}\label{sec:prob_set}
Throughout, we follow standard notation of convex optimization, as set out for example in the monographs \cite{nest_book_full,Beck17}. We let $\R^d$ denote an Euclidean space with inner product $\langle \cdot,\cdot\rangle$ and the induced norm $\|x\|=\sqrt{\langle x,x\rangle}$. The symbol $B_{\varepsilon}(x)$ will stand for the closed ball around $x$ of radius $\varepsilon>0$. We will use the shorthand interval notation $[1,m]:=\{1,\ldots,m\}$ for any number $m\in \mathbb{N}$. Abusing notation slightly, for any set of real numbers $\{r_i\}_{i=1}^m$ we will let $\texttt{median}(r_1,r_2,\ldots,r_m)$ denote the $\lceil\frac{m}{2}\rceil$'th entry in the ordered list $r_{[1]}\leq r_{[2]}\leq \ldots\leq r_{[m]}$.

Consider a function $f\colon\R^d\to\R\cup\{+\infty\}$. The effective domain of $f$, denoted $\dom f$, consists of all points where $f$ is finite.
The function $f$ is called $\mu$-strongly convex if the perturbed function $f-\frac{\mu}{2}\|\cdot\|^2$ is convex. We say that $f$ is $L$-smooth if it differentiable with $L$-Lipschitz continuous gradient. 
If $f$ is both $\mu$-strongly convex and $L$-smooth, then standard results in convex optimization (e.g., \cite[\S~2.1]{nest_book_full}) imply 
for all $x,y\in\R^d$ the bound
$$
\langle \nabla f(y),x-y\rangle+\frac{\mu}{2}\|x-y\|^2 \leq
f(x) - f(y) \leq \langle \nabla f(y),x-y\rangle+\frac{L}{2}\|x-y\|^2.
$$
In particular, 
if $y$ is the minimizer of~$f$, denoted by~$\bar x$,
we have $\nabla f(\bar x)=0$ and thus
the two-sided bound: 
\begin{equation}\label{eqn:smooth_conv_est}
\frac{\mu}{2}\|x-\bar x\|^2\leq f(x)-f(\bar x)\leq \frac{L}{2}\|x-\bar x\|^2\qquad \textrm{for all }x \in \R^d.
\end{equation}
The ratio $\kappa:=L/\mu$ is called the condition number of $f$.  


\begin{assumption}\label{ass:strong_conv}
	{\rm Throughout this work, we consider the optimization problem 
		\begin{equation}\label{eqn:targ_prob}
		\min_{x\in \R^d}~ f(x)
		\end{equation}
		where the function $f\colon\R^d\to\R\cup\{+\infty\}$
is closed and $\mu$-strongly convex. We denote the minimizer of $f$ by $\bar x$ and its minimal value by $f^*:=\min f$.}
\end{assumption}

Let us suppose for the moment that the only access to $f$ is by querying 
a black-box procedure that estimates $\bar x$. Namely following \cite{hsu2016loss} we will call a procedure $\cD(\varepsilon)$ a {\em weak distance oracle} for the problem \eqref{eqn:targ_prob} if it returns a point $x$ satisfying 
\begin{equation}\label{eqn:resampling}
\PP[\|x-\bar x\|\leq \varepsilon]\geq\frac{2}{3}.
\end{equation}
We will moreover assume that when querying $\cD(\varepsilon)$ multiple times, the returned vectors are all statistically independent. Weak distance oracles arise naturally in stochastic optimization both in streaming and offline settings. We will discuss specific examples in Sections~\ref{sec:cons_ERM} and~\ref{sec:conseq_approx}.
 The numerical value $2/3$ plays no real significance and can be replaced by any fraction greater than a half.

It is well known from \cite[p. 243]{complexity} and \cite{hsu2016loss} that the low-confidence estimate \eqref{eqn:resampling} can be improved to a high confidence guarantee by a clustering technique. 
Following \cite{hsu2016loss}, we define the {\em robust distance estimator} $\cD(\varepsilon,m)$ to be the following procedure (Algorithm~\ref{alg:stoc_prox_high_prob}).

\begin{algorithm}[H]
	{\bf Input:}  access to a weak distance oracle $\cD(\varepsilon)$ and trial count $m$. \\
	Query $m$ times the oracle $\cD(\varepsilon)$ and let $Y=\{y_1,\ldots, y_m\}$ consist of the responses.
	
	{\bf Step } $i=1,\ldots,m$:\\
	\hspace{20pt} Compute $r_i=\min\{r\geq 0: |B_{r}(y_i)\cap Y|>\frac{m}{2}\}$.
	
	Set $i^*=\argmin_{i\in [1,m]} r_i$ 

	 {\bf Return} $y_{i^*}$		\\
	\caption{Robust Distance Estimation $\cD(\varepsilon,m)$	
	}
	\label{alg:stoc_prox_high_prob}
\end{algorithm}

Thus the robust distance estimator  $\cD(\varepsilon,m)$ first generates $m$ statistically independent random points  $y_1,\ldots, y_m$ by querying $m$ times the weak distance oracle $\cD(\varepsilon)$. Then the procedure computes the smallest radius ball around each point $y_i$ that contains more than half of the generated points $\{y_1,\ldots, y_m\}$. Finally, the point $y_{i^*}$ corresponding to the smallest such ball is returned. 
See Figure~\ref{fig_rde} for an illustration.

\begin{figure}[h]
	\centering
	\includegraphics{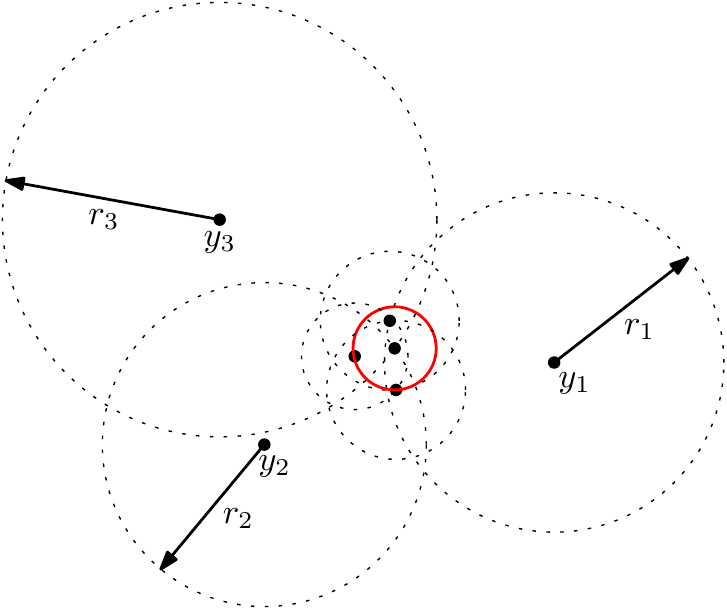}
	\caption{Illustration of the robust distance estimator $\cD(\varepsilon,m)$.}
	\label{fig_rde}
\end{figure}

The intuition underlying the algorithm is that by Chernoff's bound, with high confidence, the ball $B_{\varepsilon}(\bar x)$ will contain strictly more than $m/2$ of the generated points. Therefore in this event, the estimate $r_{i^*}<2\varepsilon$ holds.
 Moreover since the two sets, $B_{\varepsilon}(\bar x)$ and $B_{r_{i^*}}(y_{i^*})$ intersect, it follows that $\bar x$ and $y_{i^*}$ are within a distance of $3\varepsilon$ of each other. For a complete argument, see \cite[p. 243]{complexity} or \cite[Propositions 8,9]{hsu2016loss}. 
\begin{lemma}[Robust Distance Estimator]\label{lem:rob_dist_est}
	The point $x$ returned by $\cD(\varepsilon,m)$ satisfies $$\PP\bigl(\|x- \bar x\|\leq 3\varepsilon\bigr)\geq 1-\exp\left(-\frac{m}{18}\right).$$
\end{lemma}


We seek to understand how one may use a robust distance estimator $\cD(\varepsilon,m)$ to compute a point $x$ satisfying $f(x)-\min f\leq \delta$ with high probability, where $\delta>0$ is a specified accuracy. As motivation, consider the case when $f$ is also $L$-smooth.
Then one immediate approach is to appeal to the upper bound in \eqref{eqn:smooth_conv_est}. Hence by Lemma~\ref{lem:rob_dist_est}, the point $x=\cD\left(\varepsilon,m\right)$, with $\varepsilon=\sqrt{\frac{2\delta}{9L}}$, satisfies the guarantee 
$$\mathbb{P}\left(f(x)-f^*\leq \delta\right)
~\geq~ \PP\left(\|x-\bar{x}\|\leq 3\varepsilon\right) ~\geq~ 1-\exp\left(-\frac{m}{18}\right).$$

We will follow an alternative approach, which  can significantly decrease the overall cost in the regime $\kappa\gg 1$. The optimistic goal is to replace the accuracy $\varepsilon\approx\sqrt{\frac{\delta}{L}}$ used in the call to $\cD(\varepsilon,m)$ by the potentially much larger quantity $\sqrt{\frac{\delta}{\mu}}$. The strategy we propose  will apply a robust distance estimator $\cD$ to a sequence of optimization problems that are better and better conditioned, thereby amortizing the overall cost. In the initial step, we will simply apply $\cD$ to $f$ with the low accuracy 
$\sqrt{\frac{\delta}{\mu}}$. In step $i$, we will apply $\cD$ to a new function $f^{i}$, which has condition number $\kappa_i\approx\frac{L+\mu 2^i}{\mu+\mu 2^i}$, with accuracy $\varepsilon_i\approx\sqrt{\frac{\delta}{\mu+\mu 2^i}}$. Continuing this process for $T\approx \log_2\bigl(\frac{L}{\mu}\bigr)$ rounds, we arrive at accuracy $\varepsilon_T \approx \sqrt{\frac{\delta}{\mu+L}}$ and a function $f^T$ that is nearly perfectly conditioned with $\kappa_T\leq 2$. In this way, the total cost is amortized over the sequence of optimization problems. The key of course is to control the error incurred by varying the optimization problems along the iterations.

\section{Main result}\label{sec:main_res}

The continuation procedure outlined at the end of the previous section can be succinctly described within the framework of an {\em inexact proximal point method}. Henceforth, fix an increasing sequence of penalties $\lambda_0,\ldots, \lambda_T$ and a sequence of centers $x_0,\ldots,x_T$. 
 For each index $i=0,\ldots, T$, define the quadratically perturbed functions and their minimizers:
\begin{align*}
f^{i}(x):=f(x)+\frac{\lambda_i}{2}\|x-x_i\|^2,\qquad \bar x_{i+1}:=\argmin_x f^{i}(x).
\end{align*}
The exact proximal point method \cite{MR0290213,MR0298899,MR0410483} proceeds by inductively declaring $x_{i}=\bar x_{i}$ for $i\geq 1$. Since computing $\bar x_{i}$ exactly is in general impossible, we will instead monitor the error $\|\bar x_{i}-x_{i}\|$. The following elementary result will form the basis for the rest of the paper. To simplify notation, we will set  $\bar x_0:=\argmin f$ and $\lambda_{-1}:=0$, throughout.

\begin{theorem}[Inexact proximal point method]\label{thm:main_proppointest}
For all $j\geq 0$, the following estimate holds:
\begin{equation}\label{eqn:bound_opt_val_iters}
f^{j}(\bar x_{j+1})-f^*\leq  \sum_{i=0}^{j}\frac{\lambda_{i}}{2}\|\bar x_{i}-x_{i}\|^2.
\end{equation}
Consequently, we have the error decomposition:
\begin{equation}\label{eqn:func_progress}
\boxed{f(x_{j+1})-f^*\leq (f^{j}(x_{j+1})-f^{j}(\bar x_{j+1}))+\sum_{i=0}^{j} \frac{\lambda_i}{2}\|\bar x_{i}-x_{i}\|^2.}
\end{equation}
Moreover, if $f$ is $L$-smooth, then for all $j\geq 0$ the estimate holds:
\begin{align}
f(x_j)-f^*&\leq \frac{L+\lambda_{j-1}}{2}\|\bar x_{j}-x_{j}\|^2+\sum_{i=0}^{j-1}\frac{\lambda_{i}}{2}\|\bar x_{i}-x_{i}\|^2. \label{eqn:init_better}
\end{align}
\end{theorem}


\begin{proof}
    We first establish \eqref{eqn:bound_opt_val_iters} by induction. For the base case $j=0$, observe $\lambda_{-}=0$ and
$$f^{0}(\bar x_{1}) = \min_x f^{0}(x) \leq f^{0}(\bar x_{0})=f^*+\frac{\lambda_0}{2}\|\bar x_{0}-x_0\|^2.$$
As the inductive assumption, suppose \eqref{eqn:bound_opt_val_iters} holds up to iteration $j-1$. We then conclude 
\begin{align*}
    f^{j}(\bar x_{j+1}) \leq f^{j}(\bar x_{j})
&= f(\bar x_{j})+\frac{\lambda_{j}}{2}\|\bar x_{j}-x_{j}\|^2\\
&\leq f^{j-1}(\bar x_{j})+	\frac{\lambda_{j}}{2}\|\bar x_{j}-x_{j}\|^2\leq f^*+ \sum_{i=0}^j\frac{\lambda_{i}}{2}\|\bar x_i-x_i\|^2,
\end{align*}
where the last inequality follows by the inductive assumption. This completes the proof of \eqref{eqn:bound_opt_val_iters}. To see \eqref{eqn:func_progress}, we observe using \eqref{eqn:bound_opt_val_iters} the estimate  
\begin{align*}
f(x_{j+1})-f^*\leq f^{j}(x_{j+1})-f^*&=(f^{j}(x_{j+1})-f^{j}(\bar x_{j+1}))+f^{j}(\bar x_{j+1})-f^*\\ 
&\leq (f^{j}(x_{j+1})-f^{j}(\bar x_{j+1}))+\sum_{i=0}^{j}\frac{\lambda_{i}}{2}\|\bar x_{i}-x_{i}\|^2.
\end{align*}
Finally, if $f$ is $L$-smooth, then $f^j$ is $(L+\lambda_j)$-smooth. 
An analogous result to~\eqref{eqn:smooth_conv_est} yields
$$f^{j}(x_{j+1})-f^{j}(\bar x_{j+1})\leq \frac{L+\lambda_{j}}{2}\|\bar x_{j+1}-x_{j+1}\|^2.$$
Inequality \eqref{eqn:init_better} follows from applying this bound 
in \eqref{eqn:func_progress}.
\end{proof}

The main conclusion of Theorem~\ref{thm:main_proppointest} is the decomposition of the functional error  described in \eqref{eqn:func_progress}. Namely, the estimate \eqref{eqn:func_progress} upper bounds the error $f(x_{j+1})-\min f$ as the sum of the suboptimality in the last step $f^{T}(x_{T+1})-f^{T}(\bar x_{T+1})$ and the errors  
$\frac{\lambda_i}{2}\|\bar x_{i}-x_{i}\|^2$ incurred along the way. By choosing $T$ sufficiently large, we can be sure that the function $f^{T}$ is well-conditioned. Moreover in order to ensure that each term in the sum $\frac{\lambda_i}{2}\|\bar x_{i}-x_{i}\|^2$ is of order $\delta$, it suffices to guarantee $\|\bar x_{i}-x_{i}\|\leq \sqrt{\frac{2\delta}{\lambda_i}}$ for each index $i$. Since $\lambda_i$ is an increasing sequence, it follows that we may gradually decrease the  tolerance on the errors $\|\bar x_{i}-x_{i}\|$, all the while improving the conditioning of the functions we encounter. With this intuition in mind, we introduce the \pboost procedure (Algorithm~\ref{alg:stoc_prox_basic}). The algorithm, and its latter modifications, depend on the amplitude sequence $\{\lambda_j\}_{j=1}^T$ governing the proximal regularization terms. To simplify notation, we will omit this sequence from the algorithm input and instead treat it as a global parameter specified in theorems.

\begin{algorithm}[H]
    {\bf Input:} $\delta\geq 0$, $p\in (0,1)$, $T\in \mathbb{N}$\\
	Set $\lambda_{-1}=0$, $\varepsilon_{-1}=\sqrt{\frac{2\delta}{\mu}}$
	
	Generate a point $x_{0}$ satisfying $\|x_{0}-\bar x_0\|\leq \varepsilon_{-1}$ with probability $1-p$.
	
	\For{ $j=0,\ldots,T-1$}{
	Set $\varepsilon_{j}=\sqrt{\frac{2\delta}{\mu+\lambda_{j}}}$
	
	Generate a point $x_{j+1}$ satisfying 
	\begin{equation}\label{eqn:keyesneededblah}
	\PP\left[\|x_{j+1}-\bar x_{j+1}\|\leq \varepsilon_{j} \mid E_{j}\right]\geq 1-p,
	\end{equation}
	where $E_j$ denotes the event $E_j:=\left\{x_i\in B_{\varepsilon_{i-1}}(\bar x_i)\textrm{ for all }i\in [0,j]\right\}$.}

\smallskip
Generate a point $x_{T+1}$ satisfying
	\begin{equation}\label{eqn:clean_up}
	\PP\left[f^{T}(x_{T+1})-\min f^{T}\leq \delta \mid E_{T}\right]\geq 1-p.
	\end{equation}
	
	{\bf Return} $x_{T+1}$

	\caption{$\cb(\delta, p, T)$	
	}
	\label{alg:stoc_prox_basic}
\end{algorithm}

Thus  \pboost  consists of three stages, which we now examine in detail.


\paragraph{Stage I: Initialization.} 
Algorithm~\ref{alg:stoc_prox_basic} begins by generating a point $x_0$ that is a distance of $\sqrt{\frac{2\delta}{\mu}}$ away from the minimizer of $f$ with probability $1-p$. This task can be achieved by applying a robust distance estimator on $f$, as discussed in Section~\ref{sec:prob_set}.

\paragraph{Stage II: Proximal iterations.}
In each subsequent iteration, $x_{j+1}$ is defined to be a point that is within a radius of $\varepsilon_j=\sqrt{\frac{2\delta}{\mu+\lambda_{j}}}$  from the minimizer of $f^{j}$ with probability $1-p$ conditioned on the event $E_j$.
 The event $E_j$ encodes that each previous iteration was successful in the sense that the point $x_i$ indeed lies inside the ball $B_{\varepsilon_{i-1}}(\bar x_i)$ for all $i=0,\ldots, j$. Thus $x_{j+1}$ can be determined by a procedure that conditioned on the event $E_j$ is a robust distance estimator on the function $f^{j}$.
 
 \paragraph{Stage III: Cleanup.}
 In the final step, the algorithm outputs a $\delta$-minimizer of $f^T$ with probability $1-p$ conditioned on the event $E_{T}$. 
 In particular, if $f$ is $L$-smooth then we may use a robust distance estimator on $f^T$ directly. Namely, taking into account the upper bound in~\eqref{eqn:smooth_conv_est}, we may declare 
$x_{T+1}$ to be any point satisfying 
 $$\PP\left[\|x_{T+1}-\bar x_{T+1}\|\leq \sqrt{\tfrac{2\delta}{L+\lambda_T}}\mid E_{T}\right]\geq 1-p.$$
 Notice that by choosing $\lambda_T$ sufficiently large, we may ensure that the condition number $\tfrac{\mu+\lambda_T}{L+\lambda_T}$ of $f^T$ is arbitrarily close to one. If $f$ is not smooth, such as when constraints or additional regularizers are present, we can not use a robust distance estimator in the cleanup stage. We will see in Section~\ref{sec:conv_comp_ext} a different approach for convex composite problems, based on a modified robust distance estimation technique.

\bigskip

The following theorem summarizes the   guarantees of the $\cb$ procedure.
\begin{theorem}[Proximal Boost]\label{thm:conf_boost_basic}
Fix a constant $\delta>0$, a probability of failure $p\in (0,1)$ and a natural number $T\in \mathbb{N}$.
Then with probability at least $1-(T+2)p$, the point $x_{T+1}=\cb(\delta,p,T)$
	satisfies
	\begin{equation}\label{eqn:gap_bounds}
	f(x_{T+1})-\min f \leq \delta \left(1+ \sum_{i=0}^T \frac{\lambda_i}{\mu+\lambda_{i-1}}\right).
	\end{equation}
\end{theorem}
\begin{proof}
We first prove by induction the estimate 
\begin{equation}\label{eqn:induct_prob}
\PP[E_t]\geq 1-(t+1)p\qquad \textrm{for all }t=0,\ldots, T.
\end{equation}
The base case $t=0$ is immediate from the definition of $x_0$.
Suppose now that \eqref{eqn:induct_prob} holds for some index $t-1$. 
Then the inductive assumption and the definition of $x_t$ yield
$$\PP[E_t]= \PP[E_{t}\bigl\vert E_{t-1}]\PP[E_{t-1}]\geq \left(1-p\right)\left(1-tp\right)\geq 1-(t+1)p,$$
thereby completing the induction. Thus the inequalities \eqref{eqn:induct_prob} hold. Define the event $$F=\{f^{T}(x_{T+1})-\min f^T\leq \delta\}.$$
We therefore deduce  
$$\PP[F\cap E_T]=\PP[F\mid E_T]\cdot \PP[E_T]\geq (1-(T+1)p)(1-p)\geq 1-(T+2)p.$$
Suppose now that the event $F\cap E_{T}$ occurs. Then
using the estimate \eqref{eqn:func_progress}, we conclude 
\begin{equation*}
f(x_{T+1})-\min f\leq (f^{T}(x_{T+1})-f^{T}(\bar x_{T+1}))+\sum_{i=0}^{T} \frac{\lambda_i}{2}\|\bar x_{i}-x_{i}\|^2\leq \delta+\sum_{i=0}^T \frac{\delta\lambda_i}{\mu+\lambda_{i-1}}, 
\end{equation*}
where the last inequality uses the definitions of $x_{T+1}$ and  $\varepsilon_j$. This completes the proof.
%
\end{proof}

Looking  at the estimate~\eqref{eqn:gap_bounds}, we see that the final error $f(x_{T+1})-\min f$ is controlled by the sum $\sum_{i=0}^T \frac{\lambda_i}{\mu+\lambda_{i-1}}$. A moment of thought yields an appealing choice $\lambda_i=\mu 2^i$ for the proximal parameters. Indeed, then every element in the sum $\frac{\lambda_{i}}{\mu+\lambda_{i-1}}$ is upper bounded by two. Moreover, if $f$ is $L$-smooth, then the condition number $\frac{L+\lambda_T}{\mu+\lambda_T}$ of $f^T$ is upper bounded by two after only $T=\lceil\log(L/\mu)\rceil$ rounds.

\begin{corollary}[Proximal boost with geometric decay]
Fix an iteration count $T$, a target accuracy $\epsilon>0$, and a probability of failure $p\in (0,1)$. Define the algorithm parameters:
	 $$\qquad \delta=\frac{\epsilon}{2+2T}\qquad \textrm{and} \qquad \lambda_i=\mu 2^i\qquad \forall i\in [0,T].$$
	Then the point  $x_{T+1}=\cb(\delta,p,T)$
 satisfies 
	$$\PP(f(x_{T+1})-\min f\leq \epsilon)\geq 1-(T+2)p.$$ 
\end{corollary}

In the next two sections, we seed the \pboost procedure with (accelerated) stochastic gradient algorithms and methods based on empirical risk minimization. The reader, however, should keep in mind that \pboost is entirely agnostic to the inner workings of the robust distance estimators it uses. The only point to be careful about is that some distance estimators (e.g., when using stochastic gradient methods) require auxiliary quantities as input, such as an upper estimate on the function gap at the initial point. Therefore, we may have to update such estimates along the iterations of $\cb$.

\section{Consequences for empirical risk minimization}\label{sec:cons_ERM}
In this section, we explore the consequences of the \pboost algorithm for empirical risk minimization. Setting the stage, fix a probability space $(\Omega,\mathcal{F},\cP)$ and equip $\R^d$ with the Borel $\sigma$-algebra. Consider the optimization problem
\begin{equation}\label{eqn:pop_min}
\min_x f(x)=\EE_{z\sim \cP} \left[f(x,z)\right],
\end{equation}
where $f\colon\R^d\times \Omega\to \R_+$ is a measurable nonnegative function. 
A common approach to problems of the form \eqref{eqn:pop_min} is based on empirical risk minimization. Namely, one collects i.i.d.\ samples $z_1,\ldots, z_n\sim \cP$ and minimizes the empirical average
\begin{equation}\label{eqn:emp_risk}
\min_x f_S(x):=\frac{1}{n}\sum_{i=1}^n f(x,z_i).
\end{equation}
A central question is to determine the number $n$ of samples that would ensure that the minimizer $x_S$ of the empirical risk has low generalization error $f(x_S)-\min f$, with reasonably high probability. There is a vast literature on this subject; some representative works include \cite{hsu2016loss,bartlett2002rademacher,shalev2009stochastic,shalev2014understanding}. We build here on the work of Hsu-Sabato \cite{hsu2016loss}, who specifically focused on high confidence guarantees for smooth strongly convex minimization.
As in the previous sections, we let $\bar x$ be a minimizer of $f$ and define the shorthand $f^*=\min f$. 

\begin{assumption}\label{asmp:erm-smooth-conv}
{\rm 
Following \cite{hsu2016loss}, we make the following assumptions on the loss. 
\begin{enumerate}
	\item {\bf (Strong convexity)} There exist a  real $\mu>0$ and a natural number $N\in \mathbb{N}$ such that:
	\begin{enumerate}
		\item the population loss $f$ is $\mu$-strongly convex,
		\item the empirical loss $x\mapsto f_S(x)$ is $\mu$-strongly convex  with probability at least $5/6$, whenever $|S|\geq N$.
	\end{enumerate}
	\item {\bf (Smoothness)} There exist constants $L, \hat L>0$ such that:
	\begin{enumerate}
		\item for a.e. $z\sim \cP$, the loss $x\mapsto f(x,z)$ is $\hat L$-smooth,
        \item the population objective $x\mapsto f(x)$ is $L$-smooth. (It holds that $L\leq \hat L$.)
	\end{enumerate}
\end{enumerate}
In addition, we assume $f^*:=\min f >0$. 
}
\end{assumption}

The following result proved in \cite[Theorem 15]{hsu2016loss} shows that the empirical risk minimizer is a weak distance oracle for the problem \eqref{eqn:pop_min}.

\begin{lemma}\label{lem:single_est_get_it}
    Fix an i.i.d.\ sample $z_1,\ldots,z_n\sim \cP$ of size $n\geq N$. Suppose Assumption~\ref{asmp:erm-smooth-conv} holds. Then the minimizer $x_S$ of the empirical risk \eqref{eqn:emp_risk} satisfies the bound:
		$$\PP\left[\|x_S-\bar x\|\leq \sqrt{\frac{96 \hat L f^*}{n\mu^2}}~\right]\geq 2/3.$$ 
\end{lemma}

In particular, using Algorithm~\ref{alg:stoc_prox_high_prob} one may turn empirical risk minimization into a robust distance estimator for the problem \eqref{sec:cons_ERM} using a total of $mn$ samples. Let us estimate the function value at the generated point by a direct application of smoothness.
Appealing to Lemma~\ref{lem:rob_dist_est} and the two-sided bound \eqref{eqn:smooth_conv_est}, we  deduce that with probability $1-\exp(-m/18)$ the procedure will return a point $x$ satisfying 
$$f(x)\leq \left(1+\frac{432\hat L L}{n\mu^2}\right)f^*.$$
Observe that this is an estimate of {\em relative error}.
In particular, let $p\in (0,1)$ be some acceptable probability of failure and let $\gamma >0$ be a desired level of relative accuracy. Then setting  $m=\lceil 18\ln(1/p)\rceil$ and $n\geq\max\{\frac{432\hat \kappa \kappa}{\gamma},N\}$, we conclude that $x$ satisfies 
\begin{equation}\label{eqn:comp_base_erm}
\PP[f(x)\leq (1+\gamma) f^*]\geq 1-p,
\end{equation}
while the overall sample complexity  is 
\begin{equation}\label{eqn:ezest}
\left\lceil 18\ln\left(\frac{1}{p}\right)\right\rceil\cdot  \max\left\{\left\lceil\frac{432\hat\kappa \kappa}{\gamma}\right\rceil, N\right\}, 
\end{equation}
where $\hat\kappa=\hat L/\mu$ and $\kappa={L}/\mu$.
This is exactly the result \cite[Corollary 16]{hsu2016loss}. 

We will now see how to find a point $x$ satisfying~\eqref{eqn:comp_base_erm} with significantly fewer samples by embedding empirical risk minimization within  \cb. Algorithm~\ref{alg:stoc_prox_high_prob0} encodes the empirical risk minimization process on a quadratically regularized problem. Algorithm~\ref{alg:stoc_prox_high_prob1} is the robust distance estimator induced by Algorithm~\ref{alg:stoc_prox_high_prob0}. Finally, Algorithm~\ref{alg:stoc_prox_high_prob3} is the \pboost algorithm specialized to empirical risk minimization.

\begin{algorithm}[H]
	{\bf Input:}  sample count $n\in \mathbb{N}$, center $x\in\R^d$, amplitude $\lambda>0$.
	
	Generate i.i.d. samples $z_1,\ldots, z_n\sim \cP$ and compute the minimizer $\bar y$ of 
	$$\min_{y}~\frac{1}{n}\sum_{i=1}^n f(y,z_i)+\frac{\lambda}{2}\|y-x\|^2.$$
	{\bf Return} $\bar y$		\\
	\caption{$\erm(n,\lambda,x)$	
	}
	\label{alg:stoc_prox_high_prob0}
\end{algorithm}

\begin{algorithm}[H]
	{\bf Input:}  sample count $n\in \mathbb{N}$, trial count $m\in \mathbb{N}$, center $x\in\R^d$, amplitude $\lambda>0$.  \\
	Query $m$ times $\erm(n,\lambda,x)$ and let $Y=\{y_1,\ldots, y_m\}$ consist of the responses.
	
	{\bf Step } $j=1,\ldots,m$:\\
	\hspace{20pt} Compute $r_i=\min\{r\geq 0: |B_{r}(y_i)\cap Y|>\frac{m}{2}\}$.
	
	Set $i^*=\argmin_{i\in [1,m]} r_i$ 
	
	{\bf Return} $y_{i^*}$		\\
	\caption{$\rerm(n,m,\lambda,x)$	
	}
	\label{alg:stoc_prox_high_prob1}
\end{algorithm}

\begin{algorithm}[H]
	{\bf Input:} $T,m\in \mathbb{N}$, $\gamma> 0$ \\
    Set $\lambda_{-1}=0$, $x_{-1}=0$, $n_{-1}=\frac{432\hat L}{\gamma\mu}$
	
	{\bf Step } $j=0,\ldots,T$:\\
	\hspace{20pt}  $x_{j}=\rerm(n_{j-1},m,\lambda_{j-1},x_{j-1})$\\
	\hspace{20pt}	$n_{j}= 432 \left\lceil\frac{\hat L+\lambda_{j}}{\mu+\lambda_j}\left(\frac{1}{\gamma}+\sum_{i=0}^j\frac{\lambda_i}{\mu+\lambda_{i-1}} \right)\right\rceil\vee N$
	
	{\bf Return} $x_{T+1}=\rerm(\frac{L+\lambda_T}{\mu+\lambda_T}\cdot n_{T},m,\lambda_{T},x_{T})$

	\caption{$\cberm(\gamma, T,m)$	
	}
	\label{alg:stoc_prox_high_prob3}
\end{algorithm}

Using Theorem~\ref{thm:conf_boost_basic}, we can now prove the following result.

\begin{theorem}[Efficiency of $\cberm$]
	Fix a target relative accuracy $\gamma>0$ and numbers $T,m\in \mathbb{N}$.
	Then with probability at least $1-(T+2)\exp\left(-\frac{m}{18}\right)$, the point $x_{T+1}=\cberm(\gamma,T,m)$
	satisfies
	$$f(x_{T+1})-f^*\leq \left(1+\sum_{i=0}^T \frac{\lambda_i}{\mu+\lambda_{i-1}}\right)\gamma f^*.$$
\end{theorem}
\begin{proof}
    We will verify that Algorithm~\ref{alg:stoc_prox_high_prob3} is an instantiation of Algorithm~\ref{alg:stoc_prox_basic} with $\delta=\gamma f^*$ and $p=\exp(-\frac{m}{18})$. More precisely, we will prove by induction that with this choice of $p$ and $\delta$, the iterates $x_j$ satisfy \eqref{eqn:keyesneededblah} for each index $j=0,\ldots, T$ and $x_{T+1}$ satisfies \eqref{eqn:clean_up}. As the base case, consider the evaluation $x_{0}=\rerm(n_{-1},m,\lambda_{-1},x_{-1})$ where $x_{-1}$ can be arbitrary since $\lambda_{-1}=0$. Then Lemma~\ref{lem:rob_dist_est} and Theorem~\ref{lem:single_est_get_it}  guarantee
	$$\PP\left[\|x_0-\bar x_0\|\leq  3\sqrt{\frac{96 \hat L f^*}{n_{-1}\mu^2}}\right]\geq 1-\exp\left(-\frac{m}{18}\right).$$
    Taking into account the definitions of $n_{-1}$ in Algorithm~\ref{alg:stoc_prox_high_prob3} and $\varepsilon_{-1}$ in Algorithm~\ref{alg:stoc_prox_basic}, we deduce 
	$$\PP\left[\|x_0-\bar x_0\|\leq  \epsilon_{-1}\right]\geq 1-p,$$
	as claimed. As an inductive hypothesis, suppose that \eqref{eqn:keyesneededblah} holds for $x_0,x_1,\ldots, x_{j-1}$. We will prove it holds for $x_j=\rerm(n_{j-1},m,\lambda_{j-1},x_{j-1})$. To this end, suppose that the event $E_{j-1}$ occurs.
	 Then by the same reasoning as in the base case, the point $x_j$ satisfies 
	 \begin{equation}\label{eqn:getinductwork}
	\PP\left[\|x_j-\bar x_j\|\leq  3\sqrt{\frac{96 (\hat L+\lambda_{j-1}) f^{j-1}(\bar x_j)}{n_{j-1}(\mu+\lambda_{j-1})^2}}\right]\geq 1-\exp\left(-\frac{m}{18}\right).
	\end{equation}
    Now, using \eqref{eqn:bound_opt_val_iters} and the inductive assumption that $\|x_i-\bar{x}_i\|\leq\varepsilon_{i-1}=\sqrt{\frac{2\delta}{\mu+\lambda_{i-1}}}$ for all $i\in[0,j-1]$ (conditioned on $E_{j-1}$), we have
	$$f^{j-1}(\bar x_{j})-f^*\leq  \sum_{i=0}^{j-1}\frac{\lambda_{i}}{2}\|\bar x_{i}-x_{i}\|^2\leq \delta\sum_{i=0}^{j-1}\frac{ \lambda_i}{\mu+\lambda_{i-1}},$$
which, together with $\delta=\gamma f^*$, implies
\[
    f^{j-1}(\bar x_j) \leq f^* + \delta\sum_{i=0}^{j-1}\frac{ \lambda_i}{\mu+\lambda_{i-1}} = \left(1+\gamma\sum_{i=0}^{j-1}\frac{ \lambda_i}{\mu+\lambda_{i-1}}\right) f^*. 
\]
Combining this inequality with \eqref{eqn:getinductwork}, we conclude that conditioned on the event $E_{j-1}$, we have with probability $1-p$ the guarantee
	\begin{equation}\label{eqn:erm_middle}
	\frac{\mu+\lambda_{j-1}}{2}\|x_{j}-\bar x_{j}\|^2\leq \frac{432(\hat L+\lambda_{j-1})(1+\gamma\sum_{i=0}^{j-1}\frac{ \lambda_i}{\mu+\lambda_{i-1}})}{n_{j-1}(\mu+\lambda_{j-1})}\cdot f^*\leq\gamma f^*=\delta,
	\end{equation}
	where the last inequality follows from the definition of $n_{j-1}$. 
    This implies that the estimate \eqref{eqn:keyesneededblah} holds for $x_j$ with $\epsilon_{j-1}=\sqrt{\frac{2\delta}{\mu+\lambda_{j-1}}}$.
    Therefore, it holds for all iterates $x_0,\ldots, x_T$, as needed. 
    Suppose now that that event $E_T$ occurs. Then by exactly the same reasoning that led to \eqref{eqn:erm_middle}, and considering the extra factor $\frac{L+\lambda_T}{\mu+\lambda_T}$ multiplied to $n_T$ in the last call of \rerm, we have the estimate 
	$$\frac{\mu+\lambda_{T}}{2}\|x_{T+1}-\bar x_{T+1}\|^2\leq \frac{\mu+\lambda_{T}}{L+\lambda_{T}}\gamma f^*.$$
	Using smoothness, we therefore deduce $f^T(x_{T+1})-\min f^{T}\leq \gamma f^*=\delta$, as claimed.
	An application of Theorem~\ref{thm:conf_boost_basic} completes the proof.
\end{proof}

Finally, using the proximal parameters $\lambda_i=\mu 2^i$ yields the following guarantee.

\begin{corollary}[Efficiency of $\cberm$ with geometric decay]\label{cor:eff_erm}
	Fix a target relative accuracy $\gamma'>0$ and a probability of failure $p\in (0,1)$. Define the algorithm parameters:
	$$T=\left\lceil \log_{2}\left( \kappa\right)\right\rceil,\qquad m=\left\lceil 18\ln\left(\frac{T+2}{p}\right)\right\rceil, \qquad \gamma=\frac{\gamma'}{2+2T},\qquad \lambda_{i}=\mu 2^i.$$
	Then with probability of at least $1-p$, the point $x_{T+1}=\cberm(\gamma,T,m)$ satisfies $f(x^{T+1})\leq (1+\gamma')f^*$. Moreover, the total number of samples used by the algorithm is 
$$\mathcal{O}\left(\ln( \kappa)\ln\left(\frac{\ln(\kappa)}{p}\right)\cdot\max\left\{\left(1+\tfrac{1}{\gamma'}\right)\hat\kappa\ln(\kappa),N\right\}\right).$$
\end{corollary}

Notice that the sample complexity provided by Corollary~\ref{cor:eff_erm} is an order of magnitude better than \eqref{eqn:ezest} in terms of the dependence on the condition numbers $\hat\kappa$ and $ \kappa$.

\section{Consequences for stochastic approximation}\label{sec:conseq_approx}
We next investigate the consequences of \pboost for stochastic approximation. Namely, we will seed \pboost with the robust distance estimator,  induced by the stochastic gradient method and its accelerated variant. 
An important point is that the sample complexity of stochastic gradient methods depends on the initialization quality $f(x_0)-f^*$. Consequently, in order to know how many iterations are needed to reach a desired accuracy $\EE [f(x_i)]-f^*\leq \delta$, we must have available an upper bound on the initialization quality $\Delta\geq f(x_0)-f^*$. Therefore, we will have to dynamically update an estimate of the initialization quality for each proximal subproblem along the iterations of \pboostn. The following assumption formalizes this idea.

\begin{assumption}\label{ass:alg_min_orc}{\em
Consider the proximal minimization problem
$$\min_y~ \varphi_x(y):=f(y)+\frac{\lambda}{2}\|y- x\|^2,$$
Let $\Delta>0$ be a real number satisfying $\varphi_x(x)-\min \varphi_x\leq \Delta$.
We will let $\alg(\delta,\lambda,\Delta, x)$ be a procedure that returns a point $y$ satisfying 
$$\PP[\varphi_x(y)-\min \varphi_x\leq \delta]\geq\frac{2}{3}.$$}
\end{assumption}

Clearly, $\alg(\delta,\lambda,\Delta, x)$ is a minimization oracle in the sense of~\eqref{eqn:conf_bound}. Since the proximal function $\varphi_x$ is $(\mu+\lambda)$-strongly convex, it has a unique minimizer $\bar y_x$ and satisfies
$$\frac{\mu+\lambda}{2}\|y-\bar y_x\|^2 \leq\varphi_x(y)-\min\varphi_x.$$
Therefore, $\alg(\delta,\lambda,\Delta, x)$ is a weak distance oracle, in the sense that $\PP(\|y-\bar y_x\|\leq\varepsilon)\geq \frac{2}{3}$ with $\varepsilon=\sqrt{\frac{2\delta}{\mu+\lambda}}$.
Following the procedure in Section~\ref{sec:prob_set}, we may turn it into a robust distance estimator for minimizing $\varphi_x$, as long as $\Delta$ upper bounds the initialization error. We record the robust distance estimator induced by $\alg(\cdot)$ as Algorithm~\ref{alg:stoc_prox_high_prob2str}.

\begin{algorithm}[H]
	{\bf Input:} accuracy $\delta>0$, amplitude $\lambda>0$, upper bound $\Delta>0$,
	center $x\in\R^d$,\\
	\hspace{20pt} trial count $m\in \mathbb{N}$.  \\
	Query $m$ times $\alg(\delta,\lambda,\Delta,x)$ and let $Y=\{y_1,\ldots, y_m\}$ consist of the responses.
	
	{\bf Step } $j=1,\ldots,m$:\\
	\hspace{20pt} Compute $r_i=\min\{r\geq 0: |B_{r}(y_i)\cap Y|>\frac{m}{2}\}$.
	
	Set $i^*=\argmin_{i\in [1,m]} r_i$ 
	
	{\bf Return} $y_{i^*}$		\\
	\caption{\ralg$(\delta,\lambda,\Delta,x,m)$	
	}
	\label{alg:stoc_prox_high_prob2str}
\end{algorithm}

Henceforth, in addition to Assumptions~\ref{ass:strong_conv} and \ref{ass:alg_min_orc}, we assume that $f$ is $L$-smooth and set $\kappa=\frac{L}{\mu}$.
It is  then straightforward to instantiate \pboost with the robust distance estimator $\ralg$. We record the resulting procedure as Algorithm~\ref{alg:stoc_prox_high_prob_stoc_appr}.

\begin{algorithm}[H]
	{\bf Input:}  accuracy $\delta>0$, upper bound $\Delta_{\rm in}>0$,
	center $x_{\rm in}\in\R^d$, and $m,T\in\mathbb{N}$\\
	
	Set $\lambda_{-1}=0$, $\Delta_{-1}=\Delta_{\rm in}$, $x_{-1}= x_{\rm in}$
	
	{\bf Step } $j=0,\ldots,T$:\\
	\hspace{20pt}  $x_{j}=\ralg(\delta/9,\lambda_{j-1},\Delta_{j-1},x_{j-1},m)$
	
	\hspace{20pt} $\Delta_j=\delta\left(\frac{L+\lambda_{j-1}}{\mu+\lambda_{j-1}}+\sum_{i=0}^{j-1}\frac{ \lambda_i}{\mu+\lambda_{i-1}}\right)$

	{\bf Return} $x_{T+1}=\ralg(\frac{\mu+\lambda_T}{L+\lambda_T}\cdot\frac{\delta}{9},\lambda_{T},\Delta_{T},x_{T},m)$

    \caption{$\balg(\delta,\Delta_{\rm in}, x_{\rm in},T,m)$	
	}
	\label{alg:stoc_prox_high_prob_stoc_appr}
\end{algorithm}

We can now prove the following theorem on the efficiency of Algorithm~\ref{alg:stoc_prox_high_prob_stoc_appr}. The proof is almost a direct application of Theorem~\ref{thm:conf_boost_basic}. The only technical point is to verify that for all indices $j$, the quantity $\Delta_{j}$ is a valid upper bound on the initialization error $f^{j}(x_{j})-\min f^{j}$ in the event $E_{j}$ (defined in Algorithm~\ref{alg:stoc_prox_basic}).

\begin{theorem}[Efficiency of \balg]\label{thm:effbalg}
Fix an arbitrary point $x_{\rm in}\in \R^d$ and let $\Delta_{\rm in}$ be any constant satisfying $\Delta_{\rm in}\geq f(x_{\rm in})-\min f$. 
Fix natural numbers $T,m\in \mathbb{N}$. 
Then with probability at least $1-(T+2)\exp\left(-\frac{m}{18}\right)$, the point $x_{T+1}=\balg(\delta,\Delta_{\rm in}, x_{\rm in},T,m)$ satisfies 
$$f(x_{T+1})-\min f\leq \delta \left(1+\sum_{i=0}^T \frac{\lambda_i}{\mu+\lambda_{i-1}}\right).$$
\end{theorem}
\begin{proof}
	We will verify that Algorithm~\ref{alg:stoc_prox_high_prob_stoc_appr} is an instantiation of Algorithm~\ref{alg:stoc_prox_basic} with $p=\exp(-\frac{m}{18})$. More precisely, we will prove by induction that with this choice of $p$, the iterates $x_j$ satisfy \eqref{eqn:keyesneededblah} for each index $j=0,\ldots, T$ and $x_{T+1}$ satisfies \eqref{eqn:clean_up}. For the base case $j=0$, Lemma~\ref{lem:rob_dist_est} guarantees that with probability $1-p$, the point $x_0$ produced by the robust distance estimator $\ralg$ satisfies 
$$\|x_0-\bar x_{0}\|\leq 3\sqrt{\frac{2\cdot\delta/9}{\mu}}=\varepsilon_{-1}.$$
As an inductive hypothesis, suppose that \eqref{eqn:keyesneededblah} holds for the iterates $x_0,\ldots, x_{j-1}$ for some $j\geq 1$. We will prove it holds for $x_j$. To this end, suppose that the event $E_{j-1}$ occurs.
Then using \eqref{eqn:init_better} we deduce
\begin{align*}
f(x_{j-1})- f^*&\leq \frac{L+\lambda_{j-2}}{2}\|\bar x_{j-1}-x_{j-1}\|^2+\sum_{i=0}^{j-2}\frac{\lambda_{i}}{2}\|\bar x_{i}-x_{i}\|^2\notag\\
&\leq \frac{\delta(L+\lambda_{j-2})}{\mu+\lambda_{j-2}}+\sum_{i=0}^{j-2}\frac{\delta \lambda_i}{\mu+\lambda_{i-1}}=\Delta_{j-1},
\end{align*}
where the second inequality follows from $x_i\in B_{\varepsilon_{i-1}}(\bar x_i)$ with $\varepsilon_{i-1}=\sqrt{\frac{2\delta}{\mu+\lambda_{i-1}}}$ for all $i\in[0,j-1]$.
By examining the definition of $f^{j-1}$, we deduce $f^{j-1}(x_{j-1})=f(x_{j-1})$ and $\min f^{j-1} \geq \min f=f^*$, which imply 
\begin{equation}\label{eqn:verify_func_err}
f^{j-1}(x_{j-1})-\min f^{j-1}\leq f(x_{j-1})- f^*\leq \Delta_{j-1}.
\end{equation}
That is, $\Delta_{j-1}$ is an upper bound on the initial gap $f^{j-1}(x_{j-1})-\min f^{j-1}$ for all~$j$ whenever the event $E_{j-1}$ occurs. 
Moreover Lemma~\ref{lem:rob_dist_est} guarantees that conditioned on $E_{j-1}$ with probability $1-p$, the following estimate holds:
$$\|x_{j}-\bar x_{j}\|\leq 3\sqrt{\frac{2\cdot\delta/9}{\mu+\lambda_{j-1}}}=\varepsilon_{j-1}.$$
Thus the condition \eqref{eqn:keyesneededblah} holds for the iterate $x_j$,	as desired.

Now suppose that the event $E_T$ holds. Then exactly the same reasoning that led to \eqref{eqn:verify_func_err} yields the guarantee
	$f^{T}(x_{T})-\min f^{T}\leq \Delta_T$. Therefore  Lemma~\ref{lem:rob_dist_est} guarantees that with probability $1-p$ conditioned on $E_T$, we have
	$$\|x_{T+1}-\bar x_{T+1}\|\leq 3\sqrt{\frac{2}{\mu+\lambda_{T}}\cdot \frac{\delta}{9}\cdot\frac{\mu+\lambda_T}{L+\lambda_T}}=\sqrt{\frac{2\delta}{L+\lambda_{T}}}.$$
Taking into account the fact that $f^T$ is $(L+\lambda_T)$-smooth,
we therefore deduce $$\PP[f^T(x_{T+1})-\min f^T\leq \delta\mid E_T]\geq 1-p,$$ thereby establishing \eqref{eqn:clean_up}. An application of Theorem~\ref{thm:conf_boost_basic} completes the proof.
\end{proof}

When using the proximal parameters $\lambda_i=\mu 2^i$, we obtain the following guarantee.

\begin{corollary}[Efficiency of $\balg$ with geometric decay]\label{cor:last_stream}
	Fix an arbitrary point $x_{\rm in}\in \R^d$ and let $\Delta_{\rm in}$ be any upper bound $\Delta_{\rm in}\geq f(x_{\rm in})-\min f$. Fix a target accuracy $\epsilon>0$ and probability of failure $p\in (0,1)$, and set the algorithm parameters 
	 $$T=\left\lceil\log_2(\kappa)\right\rceil,\qquad m=\left\lceil18\ln\left(\frac{2+T}{p}\right)\right\rceil,\qquad \delta=\frac{\epsilon}{2+2T}, \qquad \lambda_i=\mu 2^i.$$
	Then the point  $x_{T+1}=\balg(\delta,\Delta_{\rm in}, x_{\rm in},T,m)$ satisfies 
	$$\PP(f(x_{T+1})-\min f\leq \epsilon)\geq 1-p.$$ 
	Moreover, the total number of calls to $\alg(\cdot)$ is
$$\left\lceil18\ln\left(\frac{\left\lceil 2+\log_2(\kappa)\right\rceil}{p}\right)\right\rceil\lceil 2+\log_2(\kappa)\rceil ,$$
while the initialization errors satisfy
$$\max_{i=0,\ldots,T+1}\Delta_{i}\leq \frac{\kappa+1+2\left\lceil \log_2(\kappa)\right\rceil}{2+2\left\lceil\log_2(\kappa)\right\rceil} \epsilon.$$
\end{corollary}

We now concretely describe how to use (accelerated) stochastic gradient methods as  $\alg(\cdot)$ within  $\pboost$.

\subsection*{Illustration: robust (accelerated) stochastic gradient methods}
 Following the standard literature on streaming algorithms, we suppose that the only access to $f$ is through a stochastic gradient oracle.
Namely, fix a probability space $(\Omega,\mathcal{F},\cP)$ and let $G\colon\R^d\times\Omega\to\R$ be a measurable map satisfying 
$$\EE_z G(x,z)=\nabla f(x)\qquad \textrm{and}\qquad \EE_z\|G(x,z)-\nabla f(x)\|^2\leq \sigma^2.$$
We suppose that for any point $x$, we may sample $z\in \Omega$ and compute the vector $G(x,z)$, which serves as an  unbiased estimator of the gradient $\nabla f(x)$. The performance of standard numerical methods within this model of computation is judged by their sample complexity---the number of stochastic gradient evaluations $G(x,z)$ with $z\sim \cP$ required by the algorithm to produce an approximate minimizer of the problem. 

Fix an initial point $x_{in}$ and let $\Delta_{\rm in}>0$ satisfy $\Delta_{\rm in}\geq f(x_0)-f^*$.
It is well known that an appropriately modified stochastic gradient method can generate a point $x$ satisfying 
$\EE f(x)-f^*\leq \epsilon$ with sample complexity
\begin{equation}\label{eqn:grad_desn_est}
\mathcal{O}\left(\kappa\log\left(\frac{\Delta_{\rm in}}{\epsilon}\right)+\frac{\sigma^2}{\mu\epsilon}\right).
\end{equation}
The accelerated stochastic gradient method of \cite[Multi-stage AC-SA, Proposition 6]{ghadimi2013optimal} and the simplified optimal algorithm of \cite[Restarted Algorithm C, Corollary 9]{kulunchakov2019estimate} have the substantially better sample complexity
\begin{equation}\label{eqn:accel_rate}
\mathcal{O}\left(\sqrt{\kappa}\log\left(\frac{\Delta_{\rm in}}{\epsilon}\right)+\frac{\sigma^2}{\mu\epsilon}\right).
\end{equation}
Clearly, we may use either of these two procedures as $\alg(\cdot)$ within the \texttt{proxBoost} framework. 
Indeed, using Corollary~\ref{cor:last_stream}, we deduce that the two resulting algorithms will find a point $x$ satisfying 
$$\PP[f(x)-f^*\leq \epsilon]\geq 1-p$$
with sample complexities
\begin{equation}\label{eqn:rob_rate1}
\mathcal{O}\left(\ln\left(\kappa\right)\ln\left(\frac{\ln\kappa}{p}\right)\cdot \left(\kappa\ln\left(\frac{\Delta_{\rm in}\ln(\kappa)}{\epsilon}\vee \kappa\right)+ \frac{\sigma^2\ln(\kappa)}{\mu\epsilon}\right)\right),
\end{equation}
and 
\begin{equation}\label{eqn:rob_rate2}
\mathcal{O}\left(\ln\left(\kappa\right)\ln\left(\frac{\ln\kappa}{p}\right)\cdot \left(\sqrt{\kappa}\ln\left(\frac{\Delta_{\rm in}\ln(\kappa)}{\epsilon}\vee \kappa\right)+ \frac{\sigma^2\ln(\kappa)}{\mu\epsilon}\right)\right),
\end{equation}
for the unaccelerated and accelerated methods, respectively. Thus, \pboost  endows the stochastic gradient method and its accelerated variant with high confidence guarantees at an overhead cost that is only polylogarithmic in $\kappa$ and logarithmic in $1/p$.

\section{Extension to convex composite problems}\label{sec:conv_comp_ext}
One limitation of the techniques presented in Sections~\ref{sec:cons_ERM} and \ref{sec:conseq_approx} is that the function $f$ to be minimized was assumed to be smooth. In particular, these techniques can not accommodate constraints or nonsmooth regularizers. To illustrate the difficulty, consider the task of minimizing a smooth and strongly convex function $f$ over a closed convex set $\cX$. 
The current approach heavily relies on the two-sided bound \eqref{eqn:smooth_conv_est}, which guarantees that the function gap $f(x)-f^*$ and the squared distance to the solution $\|x-\bar x\|^2$ are proportional up to multiplication by the condition number of $f$. When a constraint set $\cX$ is present, the left inequality of  \eqref{eqn:smooth_conv_est} still holds, but the right inequality is typically false. In particular, in the clean up stage of $\cb$, it is unclear how to turn low probability guarantees on the function gap to high probability guarantees using robust distance estimation.
In this section, we show how to overcome this difficulty and extend the aforementioned techniques to convex composite optimization problems. 

\subsection{Geometric intuition in the constrained case}
\label{sec:geometric-intuition}
Before delving into the details, it is instructive to first focus on the constrained setting, where no additional regularizers are present. This is the content of this section. In section \ref{sec:gen_setting}, we formally describe the algorithm for regularized convex optimization problems in full generality and prove correctness. Consequently, the reader may safely skip to Section \ref{sec:gen_setting}, without losing continuity.

 Setting the stage, 
consider the optimization problem
$$\min_x~ g(x)\quad\textrm{subject to}\quad x\in \cX,$$
where $g\colon\R^d\to\R$ is $\mu$-strongly convex and $L$-smooth and $\cX$ is a closed convex set. In line with the previous sections, let $\bar x$ be the minimizer of the problem and let $\kappa=\frac{L}{\mu}$ denote the condition number.
Suppose that we have available an algorithm $\cM(\epsilon)$ that generates a point $x_{\epsilon}$ satisfying 
\begin{equation}\label{eqn:g-weak-prob}
\PP\left(g(x_{\epsilon})-\min_{x\in\cX} g\leq \epsilon\right)\geq \frac{2}{3}.
\end{equation}
 Our immediate goal is to explain how to efficiently boost this low-probability guarantee to a high confidence outcome, albeit with the degraded accuracy $\kappa\epsilon$.
%
 
We begin as in the unconstrained setting with the two-sided bound:
\begin{equation}\label{eqn:two_side_constr}
\langle \nabla g(\bar x),x-\bar x\rangle+\frac{\mu}{2}\|x-\bar x\|^2\leq g(x)-g(\bar x)\leq \langle \nabla g(\bar x),x-\bar x\rangle+\frac{L}{2}\|x-\bar x\|^2 
\end{equation}
for all $x\in \cX$. In particular, if the minimizer $\bar x$ lies in the interior of $\cX$ the gradient $\nabla g(\bar x)$ vanishes and the estimate \eqref{eqn:two_side_constr} reduces to \eqref{eqn:smooth_conv_est}. In the more general constrained setting, however, the additive term $\langle \nabla g(\bar x),x-\bar x\rangle$ plays an important role. Note that optimality conditions at $\bar x$ immediately imply that this term is nonnegative 
$$\langle \nabla g(\bar x),x-\bar x\rangle\geq 0\qquad \textrm{for all }x\in \cX.$$
Moreover,  we see from the estimate \eqref{eqn:two_side_constr} that the point $x_{\epsilon}$ returned by $\cM(\epsilon)$, with probability $2/3$, lies in the region
\begin{equation}\label{eqn:sim_inter_ball_hyp}
\Lambda:=\left\{x\in\cX: \|x-\bar x\|\leq \sqrt{\frac{2\epsilon}{\mu}}\quad \textrm{and}\quad 0\leq \langle \nabla g(\bar x), x-\bar x\rangle \leq \epsilon\right\} .
\end{equation}
Thus  $x_{\epsilon}$ simultaneously lies in the ball around $\bar x$ of radius $\sqrt{2\epsilon/\mu}$ and is sandwiched between two parallel hyperplanes with normal $\nabla g(\bar x)$. See Figure~\ref{Figure:comparison} for an illustation.

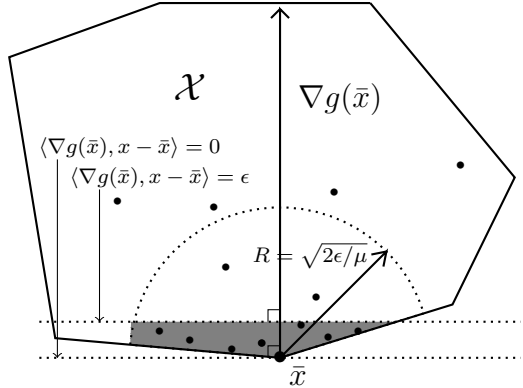
\begin{figure}[t]
	\centering
    \centering
    \begin{tikzpicture}[thick,scale=0.8]   
    \draw[dotted][fill=gray][line width=0pt] (-2.49048674523,0.21788935686-2)--(0,-2)--(1.62425554954*1.19417,1.62425554954*0.3694-2)--(-2.475,-1.4)--(-2.49048674523,0.21788935686-2);
    \draw[dotted](0,0-2) -- (2*1.19417,2*0.3694-2) arc (17.188733:175:2.5) -- cycle;
    \draw [thick](0,0-2) -- (3.5355/2,3.5355/2-2)--(3.5355/2-0.1*0.5,3.5355/2-0.5*0.5-2);
    \draw [thick] (3.5355/2,3.5355/2-2)--(3.5355/2-0.5*0.5,3.5355/2-0.1*0.5-2);
    \node at (0,0-2) {\small$\bullet$};
    \node at (0.3,-0.3-2) {$\bar x$};
    \node at (0.5,1.7-2) {\scriptsize$R=\sqrt{2\epsilon/\mu}$}; 
    \draw (0,-2) -- (2*1.19417*1.2,2*0.3694*1.2-2) -- (3.9,1)--(1.5,3.9);
    \draw (0,-2) -- (-2.49048674523*1.5,0.21788935686*1.5-2)--(-4.5,3)--(-2,3.9)--(1.5,3.9);
    \node at (-2+0.5,2.5) {\large$\cX$};
    \draw[thick](0,-2)--(0,3.8)--(0+0.15,3.8-0.15); 
    \draw (0,3.8)--(0-0.15,3.8-0.15); 
    \node at (1,2.3) {$\nabla g(\bar x)$};
    \draw[dotted] (-4,-2+0.6)--(4,-2+0.6);
    \draw[dotted] (-4,-2)--(4,-2); 
    \draw[thin] (0,-1.4+0.2)--(0-0.2,-1.4+0.2)--(-0.2,-1.4);
    \draw[thin] (0,-1.4+0.2-0.6)--(0-0.2,-1.4+0.2-0.6)--(-0.2,-1.4-0.6);
    \node at (-3+0.5,1.5) {\scriptsize$\langle\nabla g(\bar x),x-\bar x\rangle = 0$};
    \node at (-2.5+0.5,1) {\scriptsize$\langle\nabla g(\bar x),x-\bar x\rangle = \epsilon$};
    \draw [->][thin] (-3,0.8)--(-3,-1.4);
    \draw [->][thin] (-3.7,1.3)--(-3.7,-2);
    \node at (-0.3,-1.75) {\tiny$\bullet$};
    \node at (-2,-1.55) {\tiny$\bullet$};
    \node at (-0.3,-1.75) {\tiny$\bullet$};
    \node at (0.35,-1.45) {\tiny$\bullet$};
    \node at (0.8,-1.65) {\tiny$\bullet$};
    \node at (1.3,-1.55) {\tiny$\bullet$};
    \node at (-0.8,-1.85) {\tiny$\bullet$};
    \node at (-1.5,-1.7) {\tiny$\bullet$};
    \node at (0.6,-1) {\tiny$\bullet$};
    \node at (-0.9,-0.5) {\tiny$\bullet$};
    \node at (-1.1,0.5) {\tiny$\bullet$};
    \node at (-2.7,0.6) {\tiny$\bullet$};
    \node at (3,1.2) {\tiny$\bullet$};
    \node at (0.9,0.75) {\tiny$\bullet$}; 
    \end{tikzpicture}  
\caption{Geometry of the region $\Lambda$.}
\label{Figure:comparison}
\end{figure}

Naturally, our goal is to generate a point $x$ that lies in $\Lambda$, or a slight perturbation thereof, with probability $1-p$.
As the first attempt, suppose for the moment that we know the value of the gradient $\nabla g(\bar x)$. Then we can define the metric 
$$\rho(x,x')=\max\left\{\sqrt{\frac{\epsilon\mu}{2}}\cdot\|x-x'\|,~ \left|\langle \nabla g(\bar x),x-x'\rangle\right|\right\}$$
and form the robust distance estimator (Algorithm~\ref{alg:stoc_prox_high_prob}) with $\rho(\cdot,\cdot)$ replacing the Euclidean norm $\|\cdot\|$. 
In particular, the confidence bound~\eqref{eqn:g-weak-prob} and the left
inequality in~\eqref{eqn:two_side_constr} imply that $\cM(\epsilon)$ is a weak distance oracle, that is,
$\PP\bigl(\rho(x,\bar x)\leq\epsilon\bigr)\geq\frac{2}{3}.$ 
A direct extension of Lemma~\ref{lem:rob_dist_est} shows that with $m$ calls to the oracle $\cM(\epsilon)$, the robust distance estimator returns a point $x\in\cX$ satisfying
$$\PP\bigl(\rho(x,\bar x)\leq 3\epsilon\bigr)\geq 1-\exp(-m/18).$$
Consequently, appealing to the right-hand-side of \eqref{eqn:two_side_constr}, we obtain the desired guarantee
$$\PP\bigl(g(x)-g(\bar x)\leq 3(1+\kappa)\epsilon\bigr)\geq 1-p.$$

The assumption that we know the gradient $\nabla g(\bar x)$ is of course unrealistic. Therefore, the strategy we propose will instead replace the gradient $\nabla g(\bar x)$ with some estimate of the gradient $\nabla g(\hat x)$ at a nearby point $\hat x$, which we denote by $\widetilde{\nabla}g(\hat x)$. See Figure~\ref{Figure:comparison: know F' or not} for an illustration.
Indeed, a natural candidate for $\hat x$ is the robust distance estimator of $\bar x$ in the Euclidean norm. We will see that in order for the proposed procedure to work, it suffices for the gradient estimator $\widetilde{\nabla}g(\hat x)$ to approximate $\nabla g(\hat x)$ only up to the very loose accuracy $\kappa\sqrt{\mu\epsilon}$. In particular, if we have access to a stochastic gradient estimator of $\nabla g(\hat x)$ with variance $\sigma^2$, then $\widetilde{\nabla}g$ can be formed using only $\frac{1}{\kappa^2}\cdot\frac{\sigma^2}{\mu\epsilon}$ samples. This overhead in sample complexity is negligible compared to the cost of executing typical algorithms $\cM(\epsilon)$, e.g., as given in~\eqref{eqn:grad_desn_est} and~\eqref{eqn:accel_rate}.

\begin{figure}[t]
	\centering
	\begin{minipage}{0.45\textwidth} 
		\begin{tikzpicture} [scale=0.7]
		\draw [dotted][fill=gray][line width=0pt] (-4,0)--(4,0)--(3.95,0.6)--(-3.95,0.6)--(-4,0);
		\draw[dotted][thick](0,0) -- (4,0) arc (0:180:4) -- cycle;
		\draw [thick][->] (0,0) -- (4.949747/1.75,4.949747/1.75);
		\node at (0,0) {\small$\bullet$};
		\node at (0.3,-0.3) {$\bar x$};
		\node at (1.2,2.6) {\scriptsize$R=\sqrt{2\epsilon/\mu}$};
		\draw [thick][->](0,0)--(0,7.5);
		\node at (1.0,6) {\small $\nabla g(\bar x)$};
		\draw [dotted][thick] (-3.9,0.6)--(3.95,0.6);  
		\node at (2,0.3) {\small$\Lambda$}; 
		\node at (1.5,0.4) {\tiny$\bullet$};
		\node at (-0.6,1.5) {\tiny$\bullet$};
		\node at (-0.9,0.2) {\tiny$\bullet$};
		\node at (3.9,0.123) {\tiny$\bullet$};
		\node at (-2.9,0.5) {\tiny$\bullet$};
		\node at (2.6,0.15) {\tiny$\bullet$};
		\node at (1.45,2.2) {\tiny$\bullet$};
		\node at (-0.9,3.5) {\tiny$\bullet$};
		\node at (-2.9,2.5) {\tiny$\bullet$};
		\node at (-1.9,0.25) {\tiny$\bullet$};
		
		\node at (-3.8,0.55) {\tiny$\bullet$}; 
		\draw [thick][dotted] (-4.5,0)--(-4.05,0);		
        \draw [thick][dotted] (4.5,0)--(4.05,0);	
        \draw [thick][dotted] (-4.5,1.1)--(4.5,1.1);
        \draw [thin][->] (-3.8,0.55)--(-3.8,1.1);
        \draw [thin][->] (-3.8,0.55)--(-3.8,0);
        \node at (-4.5,0.25) {\scriptsize$\cO(\epsilon)$};
        \draw [thin][->] (-3.8,0.55)--(-3.8-0.5,0.55) --(-3.8-0.5,0.55 +3.8);
        \node at (-2.8, 5-0.3) {\scriptsize \begin{tabular}[c]{@{}c@{}}$\{x':\rho(x,x')\leq \epsilon\}$ contains\\more than half candidates\end{tabular}}; 
        \draw [thin] (-0,1.1+0.25)--(-0-0.25,1.1+0.25)--(-0-0.25,1.1);
        \draw [thin] (-0,0.25)--(-0-0.25,0.25)--(-0-0.25,0);
		\end{tikzpicture}  
	\end{minipage} 
	\begin{minipage}{0.45\textwidth} 
		\begin{tikzpicture} [scale=0.7]
		\draw [dotted][fill=gray][line width=0pt] (-4,0)--(4,0)--(3.95,0.6)--(-3.95,0.6)--(-4,0);
		\draw[dotted][thick](0,0) -- (4,0) arc (0:180:4) -- cycle;
		\draw [thick][->] (0,0) -- (4.949747/1.75,4.949747/1.75);
		\node at (0,0) {\small$\bullet$};
		\node at (0.3,-0.3) {$\bar x$};
		\node at (1.2,2.6) {\scriptsize$R=\sqrt{2\epsilon/\mu}$};
		\draw [thick][->](0,0)--(0,7);
		\node at (2.2,6) {\small unknown $\nabla g(\bar x)$};
		\draw [thick][->](0,0)--(-1/25*6.9,6.9);
        \node at (-2,6) {\small known $\widetilde{\nabla} g(\hat x)$};
		\draw [thin][dotted](-1/25*6.9,6.9)--(0,7);
		\node at (-0.5,7.3) {\small$\cO(\sqrt{\epsilon})$ error};
		\draw [dotted][thick] (-3.9,0.6)--(3.95,0.6);  
		\node at (2,0.3) {\small$\Lambda$}; 
		\node at (1.5,0.4) {\tiny$\bullet$};
		\node at (-0.6,1.5) {\tiny$\bullet$};
		\node at (-0.9,0.2) {\tiny$\bullet$};
		\node at (3.9,0.123) {\tiny$\bullet$};
		\node at (-2.9,0.5) {\tiny$\bullet$};
		\node at (2.6,0.15) {\tiny$\bullet$};
		\node at (1.45,2.2) {\tiny$\bullet$};
		\node at (-0.9,3.5) {\tiny$\bullet$};
		\node at (-2.9,2.5) {\tiny$\bullet$};
		\node at (-1.9,0.25) {\tiny$\bullet$};
		
		\node at (-3.8,0.55) {\tiny$\bullet$}; 
		\draw [thick][dotted] (-3.8 + 0.04*0.89,0.55 - 1*0.89)--(-3.8 + 0.04*0.89 + 1*8,0.55 - 1*0.89 + 0.04*8);	
		\draw [thick][dotted] (-3.8 + 0.04*0.89,0.55 - 1*0.89)--(-3.8 + 0.04*0.89 - 1*0.4,0.55 - 1*0.89 - 0.04*0.4);	
		\draw [thick][dotted] (-3.8 - 0.04*0.89,0.55 + 1*0.89)--(-3.8 - 0.04*0.89 + 1*8,0.55 + 1*0.89 + 0.04*8);	 
		\draw [thick][dotted] (-3.8 - 0.04*0.89,0.55 + 1*0.89)--(-3.8 - 0.04*0.89 - 1*0.4,0.55 + 1*0.89 - 0.04*0.4);	
		\draw [thin][->] (-3.8,0.55)--(-3.8 + 0.04*0.89,0.55 - 1*0.89);
		\draw [thin][->] (-3.8,0.55)--(-3.8 - 0.04*0.89,0.55 + 1*0.89);
		 \draw [thin][->] (-3.8,0.55)--(-3.8-0.5,0.55) --(-3.8-0.5,0.55 +3.8);
		\node at (-2.8, 5-0.3) {\scriptsize \begin{tabular}[c]{@{}c@{}}$\{x':\rho(x,x')\leq \epsilon\}$ contains\\more than half candidates\end{tabular}};  
		\node at (-4.5,0.25) {\scriptsize$\cO(\epsilon)$};
		\draw [thick] (0,0)--(0.04*0.15,-1*0.15);
		\draw [thin](-0.04*0.15,+1*0.15)-- (-0.04*0.15-1*0.3,+1*0.15-0.04*0.3)--
		(-0.04*0.15-1*0.3 + 0.04*0.3,+1*0.15-0.04*0.3-1*0.3);
		\draw (-0.04*1.9,+1*1.9)-- (-0.04*1.9-1*0.3,+1*1.9-0.04*0.3)--
		(-0.04*1.9-1*0.3 + 0.04*0.3,+1*1.9-0.04*0.3-1*0.3);
		\end{tikzpicture}  
	\end{minipage} 
    \caption{The left side illustrates the region $\Lambda$; the right side	depicts the perturbation of $\Lambda$ obtained by replacing the exact gradient $\nabla g(\bar x)$ with an estimator $\widetilde \nabla g(\hat x)\approx\nabla g(\hat x)$.}
	\label{Figure:comparison: know F' or not}
\end{figure}
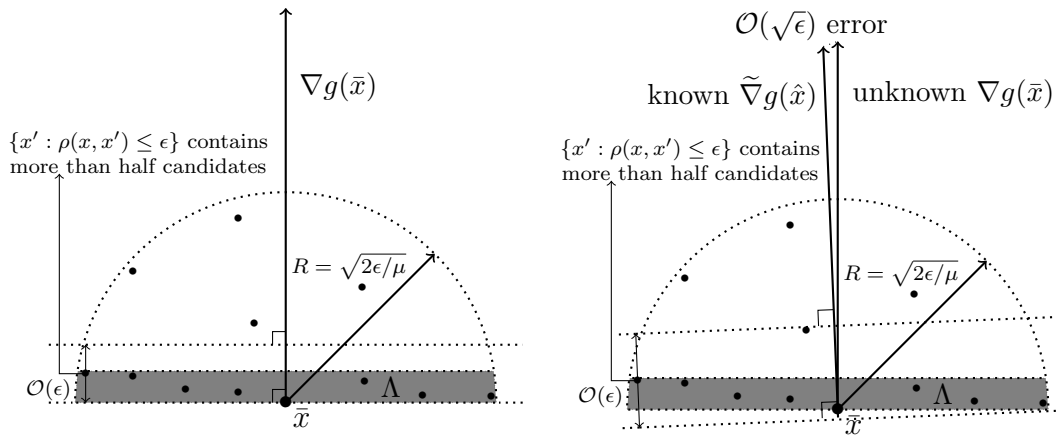

\subsection{Convex composite setting}\label{sec:gen_setting}
In this section, we formally develop the procedure that turns low probability guarantees on the function gap for composite problems to high probability outcomes. 

\begin{assumption}[Convex composite problem]\label{ass:strong_conv2}
	{\rm We consider the optimization problem 
		\begin{equation}\label{eqn:targ_prob2}
		\min_{x\in \R^d}~ f(x):=g(x)+h(x)
		\end{equation}
		where the function $g\colon\R^d\to\R$ is $L$-smooth and $\mu$-strongly convex and $h\colon\R^d\to\R\cup\{+\infty\}$ is closed and convex. We denote the minimizer of $f$ by $\bar x$, its minimal value by $f^*:=\min f$, and its condition number by $\kappa:=L/\mu$. 
}
\end{assumption}

In particular, we may model minimization of a smooth and strongly convex function $g$ over a closed convex set $\cX$ by declaring $h$ to take value zero on $\cX$ and $+\infty$ off it. 
Before stating the proposed algorithm, we require three elementary ingredients: a two-sided bound akin to \eqref{eqn:smooth_conv_est}, robust distance estimation with a ``pseudometric,'' and a robust gradient estimator.

%


\subsubsection{The two-sided bound}
Observe that optimality conditions at $\bar x$ imply the inclusion $-\nabla g(\bar x)\in \partial h(\bar x)$, where $\partial h(\bar x)$ denotes the \emph{subdifferential} of~$h$ at~$\bar x$, and therefore the nonnegativity of the quantity 
$$D_h(x,\bar x):=h(x)-h(\bar x)+\langle \nabla g(\bar x), x-\bar x\rangle.$$
Indeed, optimization specialists may recognize $D_h(x,\bar x)$ as a Bregman divergence induced by  $h$---hence the notation. 
The term $D_h(x,\bar x)$ appears naturally in a two sided bound similar to \eqref{eqn:smooth_conv_est}. 
Specifically, adding $h(x)-h(\bar x)$ to the two-sided bound~\eqref{eqn:two_side_constr} throughout, we obtain the key two-sided estimate
\begin{equation}\label{eqn:two_side_bound_reg}
\boxed{D_h(x,\bar x) +\frac{\mu}{2}\|x-\bar x\|^2 \leq f(x)-f^*\leq D_h(x,\bar x) + \frac{L}{2}\|x-\bar x\|^2.} 
\end{equation}

\subsubsection{Robust distance estimation with a pseudometric}
As the second ingredient, we will require a slight modification of the robust distance estimation technique of \cite[p. 243]{complexity} and \cite{hsu2016loss}. In particular, it will be convenient to replace the Euclidean norm $\|\cdot\|$ with a more general distance measure.

\begin{defn}[Pseudometric]{\rm
	A mapping $\rho\colon\cX\times\cX\mapsto \RR$ is a \emph{pseudometric} on a set $\cX$ if for all $x,y,z\in \cX$ it satisfies:
	\begin{enumerate}
		\item (nonnegative) $\rho(x,y)\geq 0$ and $\rho(x,x)=0$,
		\item (symmetry) $\rho(x,y) = \rho(y,x)$,
		\item (triangle inequality) $\rho(x,y)\leq \rho(x,z) + \rho(z,y)$.
		\end{enumerate}
The symbol $B_{r}^{\rho}(x)=\{y\in\cX: \rho(x,y)\leq r\}$ will denote the $r$-radius ball around $x$ in the pseudometric $\rho$.}
\end{defn}

 With this notation, we record Algorithm~\ref{alg:gen_rob_pseudo}, which is in the same spirit as the robust distance estimator, Algorithm~\ref{alg:stoc_prox_high_prob}. The differences are that the Euclidean norm is replaced with a pseudometric $\rho$, an index set is returned instead of a single point, and we leave the origin of the vectors $y_i$ unspecified for the moment.

\begin{algorithm}[h!]
	\caption{$\ext(\{y_i\}_{i=1}^m, \rho)$}\label{alg:gen_rob_pseudo}
	\label{alg:Subrountine-Robust-Distance}
	\textbf{Input:} A set of $m$ points $Y=\{y_1,...,y_m\}\subset\cX$, a pseudometric $\rho$ on $\cX$.\\  
	{\bf Step } $i=1,\ldots,m$:\\
\hspace{20pt} Compute $r_i=\min\{r\geq 0: |B_{r}^{\rho}(y_i)\cap Y|>\frac{m}{2}\}$.

Compute the median $\hat r=\texttt{median}(r_1,\ldots, r_m)$.

{\bf Return} $\mathcal{I}=\{i\in [1,m]: r_i\leq \hat r\}$.
%
\end{algorithm}

We will need the following elementary lemma, akin to Lemma~\ref{lem:rob_dist_est}. The main difference is that the lemma provides at least $m/2$ points, instead of a single point, that are close to the target with high probability. The proof is identical to that of \cite[p. 243]{complexity}  and \cite[Propositions 8 and 9]{hsu2016loss}; we provide details for the sake of completeness.

\begin{lemma}[Robust distance estimation]\label{lem:med_mean_ext}
Let $\rho$ be a pseudometric on a set $\cX$. Consider a set of points $Y=\{y_1,\ldots, y_m\}\subset\cX$ and a point $\bar y\in\cX$	satisfying $|B_{\varepsilon}^\rho(\bar y)\cap Y|>\frac{m}{2}$ for some $\varepsilon>0$. Then the index set  
$\mathcal{I}=\ext(\{y_i\}_{i=1}^m, \rho)$ satisfies the guarantee 
 $$\rho(y_{i},\bar y)\leq 3\varepsilon\quad\textrm{for all }i\in \cI.$$
\end{lemma}
\begin{proof}
	Note that for any points $y_i,y_j\in  B_{\varepsilon}^\rho(\bar y)$, the triangle inequality implies the estimate $$\rho(y_i,y_j)\leq \rho(y_i,\bar y)+\rho(\bar y,y_j)\leq 2\varepsilon.$$ 
This means that any point $y_i\in B_{\varepsilon}^{\rho}(\bar y)$, at least $\frac{m}{2}$ of them, satisfies $|B_{2\varepsilon}^\rho(y_i)\cap Y|>\frac{m}{2}$ and consequently $r_i\leq2\varepsilon$. Therefore the inequality $\hat r=\texttt{median}(r_1,\ldots, r_m)\leq 2\epsilon$ holds. 

Fix an index $i\in \mathcal{I}$. Since both $B_{r_i}^\rho(y_i)$ and $B_{\varepsilon}^\rho(\bar y)$ contain a strict majority of the points in $Y$, there must exist some point in the intersection $y\in B_{r_i}^\rho(y_i)\cap B_{\varepsilon}^\rho(\bar y)$. Using the triangle inequality, we conclude $\rho(y_i,\bar y)\leq \rho(y_i,y)+\rho(y,\bar y)\leq 3\varepsilon$, thereby completing the proof.
%
%
%
%
%
\end{proof}


%
%
%

\subsubsection{Robust gradient estimator}
\label{sec:robust-grad}
The need for this last ingredient is explained at the end of Section~\ref{sec:geometric-intuition}. Namely, we will need to estimate the gradient $\nabla g(\bar x)$, thereby perturbing the term $D_h(x,\bar x)$ in the two-sided bound~\eqref{eqn:two_side_constr}. For this purpose, we make the following mild assumption that is standard in applications. Indeed, we already encountered this assumption when paring $\pboost$ with stochastic gradient methods for unconstrained optimization.


\begin{assumption}[Stochastic first-order oracle]\label{ass:stochfirstorder}
	{\rm
		Fix a probability space $(\Omega,\mathcal{F},\cP)$ and let $G\colon\R^d\times\Omega\to\R$ be a measurable map satisfying 
		$$\EE_z G(x,z)=\nabla g(x)\qquad \textrm{and}\qquad \EE_z\|G(x,z)-\nabla g(x)\|^2\leq \sigma^2.$$
		We suppose that for any point $x$, we may sample $z\in \Omega$ and compute the vector $G(x,z)$, which serves as an  unbiased estimator of the gradient $\nabla g(x)$.}
\end{assumption}

Under this assumption, we can define a \emph{weak gradient oracle} 
$\cG_{\sigma}(\cdot,\varepsilon)$ as the averge of a finite sample of stochastic gradients, i.e., for any $\hat x\in\cX$, 
$$\cG_{\sigma}(\hat x, \varepsilon):=\frac{1}{s}\sum_{i=1}^s G(\hat x, z_i) \qquad \mbox{where}\quad s=\left\lceil\frac{3\sigma^2}{\varepsilon^2}\right\rceil.$$
Taking into account the variance reduction by a factor of~$s$ and using Markov's inequality, we have
\begin{equation*}
\PP\biggl(\biggl\|\frac{1}{s}\sum_{i=1}^s G(\hat x,z_i)-\nabla g(\hat x)\biggr\|^2\geq \varepsilon^2 \biggr)\leq \frac{\sigma^2/s}{\varepsilon^2}\leq \frac{1}{3}.
\end{equation*}
That is, $\PP\bigl(\|\cG_{\sigma}(\hat x, \varepsilon)-\nabla g(\hat x)\|<\varepsilon\bigr)\geq\frac{2}{3}$, confirming that $\cG_{\sigma}(\cdot,\varepsilon)$
is indeed a weak distance oracle in the sense of~\eqref{eqn:resampling}.
Based on this oracle, we can use Algorithm~\ref{alg:stoc_prox_high_prob} to construct a \emph{robust gradient estimator} $\cG_{\sigma}(\cdot,\varepsilon, m)$, 
which returns an estimate $\widetilde\nabla g(\cdot)$.
By Lemma~\ref{lem:rob_dist_est},
\begin{equation}\label{eqn:robust-grad-bound}
\PP\bigl(\|\widetilde\nabla g(\hat x)-\nabla g(\hat x)\|\leq 3\varepsilon\bigr)\geq 1-\exp(-m/18).
\end{equation}


\subsubsection{Robust function gap estimation}
Equipped with the three ingredients described above, we present in Algorithm~\ref{alg:Robust-Estimation-Framework} a procedure to robustly estimate the gap $f(x)-f^*$.

\SetAlgoHangIndent{4em}
\begin{algorithm}[H]
	\caption{$\rfg(\cM_f(\cdot), m, \epsilon)$.}
	\label{alg:Robust-Estimation-Framework} 
	\textbf{Input:} Minimization oracle $\cM_f(\cdot)$, integer $m\in \mathbb{N}$, accuracy $\epsilon>0$. 
	
	\textbf{Step 1:} Independently generate $x_1,...,x_m$ by the oracle $\cM_f(\epsilon)$ so that  
	\begin{equation}
	\label{eqn:f-gap-2/3}
	\PP\bigl(f(x_i)- f^*\leq \epsilon\bigr)\geq \frac{2}{3}, \quad\mbox{ for all }i\in [1,m].
	\end{equation}
	
	\textbf{Step 2:} Set $\rho_1=\|\cdot\|$ to be the usual Euclidean norm and 
	compute 
	\begin{equation}
	\cI_1:=\ext\left(\{x_i\}_{i=1}^m,\,\rho_1\right).
	\end{equation} 
	
	\textbf{Step 3:} 
	Fix arbitrary $i\in \mathcal{I}_1$ and set $\hat x:=x_i$. 
    Use the robust gradient estimator to generate $$\widetilde\nabla g(\hat x):=\cG_{\sigma}(\hat x, \kappa\sqrt{\mu\epsilon},\,m).$$
	
%
%
%
%
%
%
%
%
%
    \textbf{Step 4:} Define the pseudometric  $\rho_2(x,x'):=|h(x) - h(x')+\langle \widetilde{\nabla}g(\hat x),x-x'\rangle|$ on $\dom h$ and compute $$\cI_2= \ext(\{x_t\}_{t=1}^m,\,\rho_2).$$
	
	{\bf Return:} $x_{i}$ for an arbitrary $i\in \cI_1\cap\cI_2$.
\end{algorithm} 

Thus, the first step of $\rfg(\cM_f(\cdot), m, \epsilon)$  generates $m$ statistically independent points $x_1,\ldots,x_m$ satisfying \eqref{eqn:f-gap-2/3}. The second step determines a set of points $\{x_i\}_{i\in \cI}$ that are all close to $\bar x$ with high probability. We then choose a distinguished point $\hat x:=x_i$ for an arbitrary $i\in \cI_1$, and estimate the gradient $\nabla g(\hat x)$ with $\widetilde \nabla g(\hat x)$.
The next step approximates $D_h(\cdot,\bar x)$ with a pseudometric $\rho_2$ by replacing $\nabla g(\bar x)$ with $\widetilde \nabla g(\hat x)$, and then performs robust distance estimation to find a set of points $\{x_i\}_{i\in \cI_2}$ with low value of $D_h(x_i,\bar{x})$. Finally a point $x_i$ is returned, for any $i\in \cI_1\cap \cI_2$. The intuition is that this $x_i$ simultaneously achieves low values of $\|x_i-\bar x\|$ and $D_h(x_i,\bar x)$, thus allowing us to use~\eqref{eqn:two_side_bound_reg} for robust gap estimation.

The following theorem summarizes the guarantees of the $\rfg$ procedure. 
\begin{theorem}[Robust function gap estimation]
	\label{theorem:high-prob-bound}
	With probability at least $1-2\exp\left(-\tfrac{m}{18}\right)$, the point $x=\rfg(\cM_f(\cdot), m, \epsilon)$ satisfies the guarantee
$$\|x-\bar x\|\leq 3\sqrt{\frac{2\epsilon}{\mu}},\qquad D_h(x,\bar x)\leq 65\kappa\epsilon, \qquad f(x)-f^*\leq 74\kappa\epsilon.$$
In total, the procedure queries $m$ times the oracle $\cM_f(\epsilon)$ and evaluates $m\cdot\left\lceil\frac{3\sigma^2}{\kappa^2\mu\epsilon}\right\rceil$ times the stochastic gradient oracle $G(\hat x,\cdot)$.

\end{theorem}
\begin{proof}
Define the index set $\mathcal{J}=\{i\in [1,m]: f(x_i)-f^*\leq \epsilon\}$ and define the event
$$E:=\left\{|\mathcal{J}|>\frac{m}{2}\right\}.$$
Hoeffding's inequality for Bernoulli random variables guarantees
	$$\PP\left(E\right)\geq 1-\exp(-m/18).$$
Moreover, using the left inequality in~\eqref{eqn:two_side_bound_reg}, we deduce
\begin{equation}\label{eqn:joint_ineq_needed}
\|x_i-\bar x\|\leq \sqrt{\frac{2\epsilon}{\mu}}~~\textrm{and}~~D_h(x_i,\bar x)\leq \epsilon\qquad \textrm{for all } i\in \mathcal{J}.
\end{equation}
Henceforth, suppose that the event $E$ occurs. Then Lemma~\ref{lem:med_mean_ext} implies
\begin{equation}\label{eqn:dist_est_need}
\|x_i-\bar x\|\leq 3\sqrt{\frac{2\epsilon}{\mu}} \qquad \textrm{ for all }i\in \cI_1.
\end{equation}
As discussed in Section~\ref{sec:robust-grad}, specifically~\eqref{eqn:robust-grad-bound}, the estimate $\widetilde\nabla g(\hat x)$ generated by the robust gradient estimator $\cG_{\sigma}(\hat x, \kappa\sqrt{\mu\epsilon}, m)$ satisfies
\begin{equation}\label{eqn:get_dat_event}
\PP(\|\widetilde{\nabla}g-\nabla g(\hat x)\|\leq 3\kappa\sqrt{\mu\epsilon}~\mid ~E)\geq 1-\exp(-m/18).
\end{equation}
Define the event $\hat E:=\{\|\widetilde{\nabla}g(\hat x)-\nabla g(\hat x)\|\leq 3\kappa\sqrt{\mu\epsilon}\}$ and suppose that $E\cap \hat E$ occurs. Then, we compute
	\begin{align}
        \|\widetilde{\nabla}g(\hat x) - \nabla g(\bar x)\| & \leq \|\widetilde{\nabla}g(\hat x)-\nabla g(\hat x)\| + \|\nabla g(\hat x) - \nabla g(\bar x)\|\notag\\
	& \leq 
	 3\kappa\sqrt{\mu\epsilon} + L\|\hat x-\bar x\|\label{eqn:ver_lipgrad}\\
	& \leq 
	 3\kappa\sqrt{\mu\epsilon} + 3L\sqrt{2\epsilon/\mu} = 3(1+\sqrt{2})\kappa\sqrt{\mu\epsilon},\label{eqn:gin_grad_est}
	\end{align}
	where \eqref{eqn:ver_lipgrad} follows from \eqref{eqn:get_dat_event} and Lipschitz continuity of $\nabla g$, while \eqref{eqn:gin_grad_est} follows from \eqref{eqn:joint_ineq_needed}.
Consequently, for each index $i\in \mathcal{J}$, we successively deduce 
	\begin{align}
        \rho_2(x_i,\bar x) & =  |h(x_i)-h(\bar x)+\langle \widetilde{\nabla}g(\hat x),x_i-\bar x\rangle|\notag\\
                           & \leq  D_h(x_i,\bar x)  + |\langle \widetilde{\nabla}g(\hat x)-\nabla g(\bar x),x_i-\bar x\rangle|\notag\\
	& \leq  \epsilon + 3(1+\sqrt{2})\kappa\sqrt{\mu\epsilon}\cdot \sqrt{2\epsilon/\mu}\label{eqn:rhotwo}\\
	& =  (1+(3\sqrt{2}+6)\kappa)\epsilon, \notag
	\end{align}
	where \eqref{eqn:rhotwo} follows from \eqref{eqn:joint_ineq_needed} and \eqref{eqn:gin_grad_est}.
	Therefore, appealing to Lemma~\ref{lem:rob_dist_est} in the event $E\cap \hat E$, we conclude 
\begin{equation}\label{eqn:get_breg}
\rho_2(x_{i},\bar x)\leq 3(1+(3\sqrt{2}+6)\kappa)\epsilon \qquad \textrm{ for all }i\in \cI_2.
\end{equation}

Finally, fix an arbitrary index $i\in \cI_1\cap \cI_2$. We therefore deduce 
\begin{align}
D_{h}(x_i,\bar x)&\leq \rho_2(x_{i},\bar x) + |\langle \nabla g(\bar x)-\widetilde{\nabla}g, x_{i}-\bar x\rangle| \notag\\
&\leq 3(1+(3\sqrt{2}+6)\kappa)\epsilon + 3(1+\sqrt{2})\kappa\sqrt{\mu\epsilon}\cdot 3\sqrt{\frac{2\epsilon}{\mu}}\label{eqn:just0fy_breg_fin} \\
&= 3(1+(6\sqrt{2}+12)\kappa)\epsilon\leq  65\kappa\epsilon\label{eqn:improved},
\end{align}
where  \eqref{eqn:just0fy_breg_fin} follows from the estimates \eqref{eqn:dist_est_need}, \eqref{eqn:gin_grad_est}, and \eqref{eqn:get_breg}. 
Using the right side of ~\eqref{eqn:two_side_bound_reg}, we therefore conclude
  \begin{equation*}
   f(x_{i}) - f(\bar x)  \leq   D_{h}(x_i,\bar x)+\frac{L}{2}\|x_{i}-\bar x\|^2 
     \leq   65\kappa\epsilon+9\kappa\epsilon=74\kappa\epsilon,
   \end{equation*}
where the last inequality follows from the estimates \eqref{eqn:dist_est_need} and \eqref{eqn:improved}. Noting 
$$\PP(E\cap \hat E)=\PP(\hat E\mid E)\PP(E)\geq \left(1-\exp\left(-\tfrac{m}{18}\right)\right)\left(1-\exp\left(-\tfrac{m}{18}\right)\right)\geq 1-2\exp\left(-\tfrac{m}{18}\right),$$
 completes the proof. 
\end{proof}

With Theorem~\ref{theorem:high-prob-bound} at hand, we can now replace robust distance estimation with $\rfg$ within the $\cb$ framework, thereby making $\cb$ applicable to convex composite problems. The following two sections illustrate the consequences of the resulting method for regularized empirical risk minimization and (proximal) stochastic approximation algorithms.

\subsection{Consequences for empirical risk minimization}
In this section, we explore the consequences of $\rfg$ and $\cb$ for regularized empirical risk minimization. In particular, we will boost the low-probability guarantees developed in the seminal work \cite{shalev2009stochastic} for strongly convex problems to high confidence outcomes. The following assumption summarizes the setting of this section.

\begin{assumption}
	\label{assumption:lip-scvx-bounded}{\rm 
	Fix a probability space $(\Omega,\mathcal{F},\cP)$ and equip $\R^d$ with the Borel $\sigma$-algebra.  Throughout, we consider the optimization problem 
$$\min_x~ f(x):= g(x) + h(x)\qquad \textrm{where}\qquad g(x)=\EE_{z\sim\cP}[g(x,z)],$$
under the following assumptions. 
		\begin{enumerate}
		\item {\bf (Measurability)} 
		The function $g\colon\R^d\times\Omega\to\R$ is measurable. 
			\item {\bf (Strong convexity)} 
				The function $h\colon\R^d\to\R\cup\{+\infty\}$ is convex and there exists $ \mu>0$ such that the function $g(x,z)+h(x)$ is $\mu$-strongly convex for a.e. $z\sim\cP$. 
			\item {\bf (Lipschitz continuity)} There exists a measurable map $\ell\colon\Omega\to\R$ and a real $\bar \ell>0$ satisfying the moment bound $\sqrt{\EE_z\ell(z)^2}\leq \bar \ell$ and the Lipschitz condition 
		$$|g(x,z)-g(y,z)|\leq \ell(z)\|x-y\|\qquad \forall x\in U, z\in \Omega,$$		where $U$ is some open neighborhood of $\dom h$.
			\item {\bf (Smoothness)}   The function $g\colon\R^d\to\R$ is $L$-smooth.	
		\end{enumerate} 
	}
\end{assumption}

The first three assumptions are slight modifications of those used in \cite{shalev2009stochastic}, while the additional smoothness assumption on~$g$ will be necessary in the sequel to obtain high-confidence guarantees. That being said, the sample efficiency will depend only polylogarithmically on $L$, and therefore we will be able to treat nonsmooth loss function $g(x,z)$ using standard smoothing techniques. Under the first three assumptions, the authors of \cite{shalev2009stochastic} obtained the following guarantee for the accuracy of empirical risk minimization.

\begin{lemma}[{\cite[Theorem 6]{shalev2009stochastic}}]
	\label{lemma:con-gen-bound}
	Let the set $S\subset\Omega$ consist of $n$ i.i.d. samples drawn from $\cP$. Then the minimizer of the regularized empirical risk 
	$$x_S:=\argmin_x~ \frac{1}{n}\sum_{z\in S} g(x,z)+h(x)$$
	satisfies the generalization bound
    $$\EE_S[f(x_{S})-f^*]\leq  \frac{2 {\bar l}^2}{\mu n}.$$ 
\end{lemma} 

We will see now how to equip this guarantee with a high confidence bound using $\cb$. Recall that to apply the $\rfg$ algorithm, we require an unbiased gradient estimator $G\colon \R^d\times \Omega\to\R^d$ for $g$. Let us therefore simply declare
$$G(x,z):=\nabla g(x,z).$$
Then we can upper-bound the variance by the second moment
$$\EE_z \|G(x,z)-\nabla g(x)\|^2\leq 2(\EE_z\|\nabla g(x,z)\|^2+\EE_z\|\nabla g(x)\|^2)\leq 4\bar l^2.$$
%
We are now ready to present Algorithm~\ref{alg:stoc_prox_high_prob_stoc_appr2erm} as an instantiation of the $\cb$ procedure for \emph{regularized} empirical risk minimization.
In particular, we can still use $\erm()$ in Algorithm~\ref{alg:stoc_prox_high_prob0} for the proximal subproblem, with the definition $f(y,z_i):=g(y,z_i)+h(y)$. The robust distance estimator $\rerm()$ in Algorithm~\ref{alg:stoc_prox_high_prob1} can be used without any change.

\begin{algorithm}[H]
	{\bf Input:}  accuracy $\delta>0$,	iterations $m,T\in\mathbb{N}$\\
	
    Set $\lambda_{-1}=0$, $x_{-1}=0$
	
	{\bf Step } $j=0,\ldots,T$:\\

	\hspace{20pt}  $x_{j}=\rerm\left(\frac{54 \bar l^2}{(\mu+\lambda_{j-1}) \delta},\,m,\,\lambda_{j-1},\,x_{j-1}\right)$\\

	

	 \smallskip
	 
	Define the minimization oracle $\cM_{T}(\epsilon):=\erm\left(\frac{6 \bar l^2}{(\mu+\lambda_{T})\delta},\,\lambda_{T},\, x_{T}\right).$
	 
	 
%
%
%
	 

	 {\bf Return} $x_{T+1}=\rfg\left(\cM_{T}\left(\frac{\delta( \mu+\lambda_T)}{222(L+\lambda_T)}\right),\,m,\,\frac{\delta( \mu+\lambda_T)}{222(L+\lambda_T)}\right)$

	\caption{$\cbermc(\delta,T,m)$	
	}
	\label{alg:stoc_prox_high_prob_stoc_appr2erm}
\end{algorithm}

Notice that in Algorithm~\ref{alg:stoc_prox_high_prob_stoc_appr2erm}, we only need to call $\rfg$ in the last cleanup stage, since the intermediate iterations of $\cb$ only rely on distance estimates to the optimal solutions and not on the function gap.
The following theorem and its corollary are immediate consequences of Theorems~\ref{thm:conf_boost_basic} and \ref{theorem:high-prob-bound}.

\begin{theorem}[Efficiency of $\cbermc$]
	\label{theorem:eff_erm_cons_reg}
	Fix $\delta>0$ and integers $T,m\in \mathbb{N}$.
	Then with probability at least $1-(T+3)\exp\left(-\frac{m}{18}\right)$, the point $x_{T+1}=\cbermc(\delta,T,m)$
	satisfies
	$$f(x_{T+1})-f^*\leq \left(1+\sum_{i=0}^T \frac{\lambda_i}{\mu+\lambda_{i-1}}\right)\delta.$$ 
\end{theorem}

\begin{corollary}[Efficiency of $\cbermc$ with geometric decay]\label{cor:eff_erm_cons_reg}
	Fix a target  accuracy $\epsilon>0$ and a probability of failure $p\in (0,1)$. Define the algorithm parameters:
	$$T=\left\lceil \log_{2}\left( \kappa\right)\right\rceil,\qquad m=\left\lceil 18\ln\left(\frac{T+3}{p}\right)\right\rceil, \qquad \delta=\frac{\epsilon}{4+2T},\qquad \lambda_{i}=\mu 2^i.$$
	Then the point $x_{T+1}=\cbermc(\delta,T,m)$ satisfies 
	$$\PP(f(x^{T+1})-f(x^*)\leq \epsilon)\geq 1-p.$$
	Moreover, the total number of samples used by the algorithm is 
	\begin{equation}\label{eqn:final_constrained_ERM}
	\mathcal{O}\left(\ln^2( \kappa)\ln\left(\frac{\ln(\kappa)}{p}\right)\cdot\frac{\bar{\ell}^2}{\epsilon\mu} \right).
	\end{equation}
\end{corollary}

Thus, \pboost  endows regularized empirical risk minimization  with high confidence guarantees at an overhead cost that is only polylogarithmic in $\kappa$ and logarithmic in $1/p$.
In particular, observe that the sample complexity \eqref{eqn:final_constrained_ERM} established in Corollary~\ref{cor:eff_erm_cons_reg} depends on the smoothness parameter $L$ (through $\kappa=L/\mu$) only polylogarithmically. Consequently, it appears plausible that if the losses $g(\cdot,z)$ are nonsmooth, we may simply replace them by a smooth approximation and apply \cbermc. The price to pay should then only be polylogarithmic in the target accuracy $\epsilon$. Let us formally see how this can be done. To this end, we will assume that the optimization problem in question is to minimize a sum of an expectation of convex functions, a deterministic smooth and strongly convex function (e.g. squared $\ell_2$ norm), and a nonsmooth regularizer.

\begin{assumption}\label{ass:innonsmooth}
{\rm
Consider the optimization problem 
\begin{equation}\label{eqn:target_nonsmooth}
\min_x~ f(x):= \EE_{z\sim\cP}[g(x,z)] + \varphi(x)+ h(x)
\end{equation}
under the following assumptions. 
		\begin{enumerate}
		\item {\bf (Measurability)} 
		The function $g\colon\R^d\times\Omega\to\R$ is measurable and the assignment $x\mapsto g(x,z)$ is convex for a.e. $z\in \Omega$. 
			\item {\bf (Strong convexity)} 
				The  function $h\colon\R^d\to\R\cup\{+\infty\}$ is convex and there exist parameters $\mu,\beta>0$ such that the function $\varphi\colon\R^d\to\R$ is $\mu$-strongly convex and $\beta$-smooth for a.e. $z\in \Omega$. 
			\item {\bf (Lipschitz continuity)} There exists a measurable map $\ell\colon\Omega\to\R$ and a real $\bar \ell>0$ satisfying the moment bound $\sqrt{\EE_z\ell(z)^2}\leq \bar \ell$ and the Lipschitz condition 
		$$|g(x,z)-g(y,z)|\leq \ell(z)\|x-y\|\qquad \forall x\in \R^d, z\in \Omega.$$		
		\end{enumerate} }
\end{assumption}

%
%
%
%

The strategy we follow is to simply replace $g(\cdot,z)$ by a smooth approximation and then apply $\cbermc$. We now make precise what we mean by a smooth approximation. Assumptions of this type are classical in convex optimization; see for example Nesterov \cite{smooth_min_nonsmooth} and Beck-Teboulle \cite{smoothing_beckT}.

\begin{assumption}[Smoothing]\label{ass:smoothing}{\rm
Suppose that for any parameter $\epsilon>0$, there exist measurable functions $g_{\epsilon}\colon\R^d\times\Omega\to\R$ and $\ell_{\epsilon},L_{\epsilon}\colon\Omega\to\R_+$ such that $g_{\epsilon}(\cdot,z)$ is Lipschitz continuous with constant $\ell_{\epsilon}(z)$ and its gradient is Lipschitz continuous with constant $L_{\epsilon}(z)$,
and the estimate holds:
\begin{equation}\label{eqn:uniclose}
|g(x,z)-g_{\epsilon}(x,z)|\leq \epsilon\qquad \textrm{for all } x\in\R^d,z\in \Omega.
\end{equation}
We suppose moreover that the moment conditions, $\sqrt{\EE_z\ell^2_{\epsilon}(z)}\leq \bar\ell_{\epsilon}$ and $\EE_z L_{\epsilon}(z)\leq \bar L_{\epsilon}$, hold for some constants  $\bar \ell_{\epsilon},\bar L_{\epsilon}>0$.}
\end{assumption}

Let us look at two standard examples of smoothings of convex functions.

\begin{example}[Moreau envelope]\label{exa:moreau}
	{\rm	
A classical approach to smoothing a convex function is based on the Moreau envelope \cite{MR0201952}. Namely, fix a convex function $\psi\colon\R^d\to\R$. The Moreau envelope of $\psi$ with parameter $\nu>0$ is defined to be 
$$M^{\psi}_{\nu}(x)=\min_{y} ~\psi(y) + \frac{1}{2\nu}\|y-x\|^2.$$
It is well-known that $M^{\psi}_{\nu}$ is $\frac{1}{\nu}$-smooth. Moreover if $\psi$ is Lipschitz continuous with constant $\lip(\psi)$, then 
$M^{\psi}_{\nu}$ is also Lipschitz continuous with the same constant and the bound holds:
$$0\leq\psi(x)-\psi_{\nu}(x)\leq \nu(\lip(\psi))^2.$$
Coming back to our target problem \eqref{eqn:target_nonsmooth}, we may define $g_{\epsilon}(\cdot,z)$ to be the Moreau envelope of 
$g(\cdot,z)$ with parameter $\nu(z):=\frac{\epsilon}{l(z)^2}$, or more explicitly
$$g_{\epsilon}(x,z)=\min_{y} ~g(y,z) + \frac{l(z)^2}{2\epsilon}\|y-x\|^2.$$
Then the parameters from Assumption~\ref{ass:innonsmooth} become $\ell_{\epsilon}(z):=\ell(z)$ and $L_{\epsilon}(z):=\frac{l(z)^2}{\epsilon}$.}
\end{example}

\begin{example}[Compositional smoothing]\label{ex:comp_smooth}
{\rm
Often, the Moreau envelope of $g(\cdot,z)$ may be difficult to compute explicitly. In typical circumstances, however, the function $g(\cdot,z)$ may be written as a composition of a simple nonsmooth convex function with a linear map. It then suffices to replace only the outer function with its Moreau envelope---a technique famously explored by Nesterov \cite{smooth_min_nonsmooth}.

To illustrate on a concrete example, suppose that the population data consists of tuples $z=(a,b)\sim \cP$ and the loss takes the form $g(x,z)=h(\langle a,x\rangle,b)$ for some measurable function $h(\cdot,\cdot)$ that is convex and $1$-Lipschitz in its first argument. In order to control the Lipschitz constant, suppose also the moment bound $\sqrt{\EE_a\|a\|^2}\leq A$ for some constant $A>0$. Let us now define the smoothing
$$g_{\epsilon}(x,z)=h_{\epsilon}(\langle a,x\rangle,b),$$
where $h_{\epsilon}(\cdot,b)$ is the Moreau envelope of $h(\cdot,b)$ with parameter $\nu=\epsilon$. It is straightforward to verify that the estimate \eqref{eqn:uniclose} holds and that we may set 
 $\ell_{\epsilon}(z)=A$ and $L_{\epsilon}(z)=\frac{A^2}{\epsilon}$.}
\end{example}

With Assumptions~\ref{ass:innonsmooth} and \ref{ass:smoothing} at hand, we may now simply apply 
$\cbermc$ to the smoothed problem 
$$\min_x~ f_{\epsilon}(x):= g(x) + h(x)\qquad \textrm{where}\qquad g(x)=\EE_{z\sim\cP}[g_{\epsilon}(x,z)]+\varphi(x).$$
Using Corollary~\ref{cor:eff_erm_cons_reg}, we deduce that the procedure will find a point $x$ satisfying 
$$\PP(f_{\epsilon}(x)-f^*_{\epsilon}\leq \epsilon)\geq 1-p,$$
using 
$$	\mathcal{O}\left(\ln^2\left( \tfrac{\bar L_{\epsilon}+\beta}{\mu}\right)\ln\left(\tfrac{\ln(\frac{\bar L_{\epsilon}+\beta}{\mu})}{p}\right)\cdot\frac{\bar{\ell}^2_{\epsilon}}{\epsilon\mu} \right)$$
samples. Observe that with probability $1-p$ the returned point $x$ satisfies:
$$f(x)-f^*\leq f_{\epsilon}(x)-f^*_{\epsilon}+(f(x)-f_{\epsilon}(x))+(f^*_{\epsilon}-f^*)\leq 3\epsilon.$$ 
In particular, in the setup of Examples~\ref{exa:moreau} and \ref{ex:comp_smooth}, the sample complexities become:
$$\mathcal{O}\left(\ln^2\left( \tfrac{\bar \ell^2/\epsilon+\beta}{\mu}\right)\ln\left(\tfrac{\ln(\frac{\bar \ell^2/\epsilon+\beta}{\mu})}{p}\right)\cdot\frac{\bar{\ell}^2}{\epsilon\mu} \right)\quad \textrm{and}\quad \mathcal{O}\left(\ln^2\left( \tfrac{A^2/\epsilon+\beta}{\mu}\right)\ln\left(\tfrac{\ln(\frac{A^2/\epsilon+\beta}{\mu})}{p}\right)\cdot\frac{A^2}{\epsilon\mu} \right),$$
respectively. Hence, the price to pay for nonsmoothness is only polylogarithmic in $1/\epsilon$.

\subsection{Consequences for stochastic approximation}
We now extend the results of Section~\ref{sec:conseq_approx} to the convex composite setting. In addition to Assumptions~\ref{ass:strong_conv2} and \ref{ass:stochfirstorder}, in this section we will use the following composite analogue of Assumption~\ref{ass:alg_min_orc}. At the end of the section, we will let $\alg(\cdot)$ be the (accelerated) proximal stochastic gradient method.

\begin{assumption}\label{ass:alg_min_orc2}{\em
Consider the proximal minimization problem
$$\min_y~ \varphi_x(y):=g(y)+\frac{\lambda}{2}\|y- x\|^2+h(y),$$
Let $\Delta>0$ be a real number satisfying $\varphi_x(x)-\min \varphi_x\leq \Delta$.
We will let $\alg(\delta,\lambda,\Delta, x)$ be a procedure that returns a point $y$ satisfying 
$$\PP[\varphi_x(y)-\min \varphi_x\leq \delta]\geq\frac{2}{3}.$$}
\end{assumption}

The following algorithm is a direct extension of $\balg$ (Algorithm~\ref{alg:stoc_prox_high_prob2str}) to the convex composite setting; the only difference is that $\balgc$ (Algorithm~\ref{alg:Robust-Estimation-Framework}) replaces the distance estimator $\ralg(\cdot)$ with $\rfg(\cdot)$.

\begin{algorithm}[H]
    {\bf Input:}  accuracy $\delta>0$, upper bound $\Delta_{\rm in}>0$,
    initial $x_{\rm in}\in\R^d$, numbers $m,T\in\mathbb{N}$\\
	
    Set $\lambda_{-1}=0$, $\Delta_{-1}=\Delta_{\rm in}$, $x_{-1}= x_{\rm in}$
	
	{\bf Step } $j=0,\ldots,T$:\\
	\hspace{20pt} Define the minimization oracle for the proximal subproblem $$\cM_{j-1}(\cdot):=\alg(\cdot,\,\lambda_{j-1},\,\Delta_{j-1},\, x_{j-1}).$$\\	
	\hspace{20pt} Set $x_j=\rfg\bigl(\cM_{j-1}(\delta/9),\, m,\, \delta/9\bigr)$\\
	
	\hspace{20pt} Set $\Delta_j=\delta\left(9\cdot\frac{L+\lambda_{j-1}}{\mu+\lambda_{j-1}}+\sum_{i=0}^{j-1}\frac{ \lambda_i}{\mu+\lambda_{i-1}}\right)$\\

	 
	 {\bf Return} $x_{T+1}=\rfg\left(\cM_{T}\left(\frac{\delta(\mu+\lambda_T)}{74(L+\lambda_T)}\right),\,m,\,\frac{\delta(\mu+\lambda_T)}{74(L+\lambda_T)}\right)$
						
     \caption{$\balgc(\delta,\Delta_{\rm in}, x_{\rm in},T,m)$	
	}
	\label{alg:stoc_prox_high_prob_stoc_appr2}
\end{algorithm}

In $\balgc$, we need to use $\rfg$ in every proximal iteration as well as the cleanup step, because the stochastic proximal gradient method encoded as $\alg$ typically requires robust gap estimation on the initialization gap $\Delta_j$.  
The proof of the following theorem is almost identical to that of Theorem~\ref{thm:effbalg}, with  Theorem~\ref{theorem:high-prob-bound} playing the role of Lemma~\ref{lem:rob_dist_est}. 

\begin{theorem}[Efficiency of \balgc]
 Fix an arbitrary point $x_{\rm in}\in \R^d$ and let $\Delta_{\rm in}$ be any upper bound $\Delta_{\rm in}\geq f(x_{\rm in})-\min f$. Fix natural numbers $T,m\in \mathbb{N}$. Then with probability at least $1-2(T+2)\exp\left(-\frac{m}{18}\right)$, the point $x_{T+1}=\balgc(\delta,\Delta_{\rm in}, x_{\rm in},T,m)$
satisfies 
$$f(x_{T+1})-\min f\leq \delta \left(1+\sum_{i=0}^T \frac{\lambda_i}{\mu+\lambda_{i-1}}\right).$$
\end{theorem}

When using the proximal parameters $\lambda_i=\mu 2^i$, we obtain the following guarantee, which generalizes Corollary~\ref{cor:last_stream} to the composite setting.

\begin{corollary}[Efficiency of $\balgc$ with geometric decay]\label{cor:last_stream2}
	Fix an arbitrary point $x_{\rm in}\in \R^d$ and let $\Delta_{\rm in}$ be any upper bound $\Delta_{\rm in}\geq f(x_{\rm in})-\min f$. Fix a target accuracy $\epsilon>0$ and probability of failure $p\in (0,1)$, and set the algorithm parameters 
	 $$T=\left\lceil\log_2(\kappa)\right\rceil,\qquad m=\left\lceil18\ln\left(\frac{4+2T}{p}\right)\right\rceil,\qquad \delta=\frac{\epsilon}{2+2T}, \qquad \lambda_i=\mu 2^i.$$
	Then the point  $x_{T+1}=\balg(\delta,\Delta_{\rm in}, x_{\rm in},T,m)$ satisfies 
	$$\PP(f(x_{T+1})-\min f\leq \epsilon)\geq 1-p.$$ 
	Moreover, the total number of calls to $\alg(\cdot)$ is
$$\left\lceil18\ln\left(\frac{ 4+2\left\lceil\log_2(\kappa)\right\rceil}{p}\right)\right\rceil\lceil 2+\log_2(\kappa)\rceil ,$$
the number of evaluations of the stochastic gradient oracle $G(\cdot,\cdot)$ is at most\footnote{For the middle term, we use the observation,
$\displaystyle\min_{\lambda\geq 0} \frac{(L+\lambda)^2}{\mu+\lambda}\geq \frac{L}{2}$, which is straightforward to verify.}
$$\left\lceil18\ln\left(\frac{\left\lceil 4+2\log_2(\kappa)\right\rceil}{p}\right)\right\rceil\cdot\left\lceil \frac{6\sigma^2}{L\epsilon}\cdot (2+2\lceil\log_2(\kappa)\rceil)\right\rceil \cdot\lceil 2+\log_2(\kappa)\rceil,$$
and the initialization errors satisfy
$$\max_{i=0,\ldots,T+1}\Delta_{i}\leq \frac{9\kappa+1+2\left\lceil \log_2(\kappa)\right\rceil}{2+2\left\lceil\log_2(\kappa)\right\rceil} \epsilon.$$
\end{corollary}

In particular, the stochastic gradient method and its accelerated variant \cite{ghadimi2013optimal,kulunchakov2019estimate} admit proximal extensions with exactly the same sample complexities as in the smooth case, \eqref{eqn:grad_desn_est} and \eqref{eqn:accel_rate}, respectively.  Clearly, we may use either of these two procedures as $\alg(\cdot)$ within Algorithm~\ref{alg:stoc_prox_high_prob_stoc_appr2}.  Corollary~\ref{cor:last_stream2} then immediately shows that the two resulting algorithms will find a point $x$ satisfying 
$\PP[f(x)-f^*\leq \epsilon]\geq 1-p$
with the same sample complexities as in the smooth setting, \eqref{eqn:rob_rate1} and \eqref{eqn:rob_rate2}, respectively.

\bibliographystyle{plain}
\bibliography{bibliography}

\end{document}